\documentclass[11pt,reqno]{amsart}

\usepackage{amsmath,amsthm,verbatim,amssymb,amsfonts,amscd, graphicx}
\usepackage{graphics}
\usepackage{geometry}
\usepackage{tikz-cd}
\theoremstyle{plain}

\newtheorem{conjecture}{Conjecture} 
\newtheorem{defn}{Definition}[section]
\newtheorem{question}[conjecture]{Question}
\newtheorem{thm}[defn]{Theorem}
\newtheorem{cor}[defn]{Corollary}
\newtheorem{construction}[defn]{Construction}

\newtheorem{prop}[defn]{Proposition}

\newtheorem{notation}[defn]{Notation}
\newtheorem{remark}[defn]{Remark}
\newtheorem{lm}[defn]{Lemma}
\newtheorem{fact}[defn]{Fact}

\newtheorem{ingredient}{Ingredient}

\newtheorem*{sketch}{The structure of Section \ref{section 5}}
\newtheorem*{process}{Theorem \ref{process}}

\newtheorem*{extendable}{Theorem \ref{extendable}}
\newtheorem*{additional assumption}{Additional assumption}
\newtheorem*{sketch of the proof}{The sketch of the proof of Proposition \ref{resulting}}
\newtheorem*{postscript}{Postscript}

\usepackage{extarrows}
\usepackage{mathrsfs}
\usepackage{galois}
\usepackage[colorlinks,citecolor = red, linkcolor=blue,hyperindex,CJKbookmarks]{hyperref}
\usepackage[all]{hypcap}
\usepackage[all]{xy}
\usepackage{latexsym}
\usepackage{xspace}
\usepackage{subfigure}

\begin{document}
	
\title[Left orderability and foliations for
$3$-manifolds with sphere boundary]{Left orderability, 
	foliations and
	transverse $(\pi_1,\mathbb{R})$ structures for
	$3$-manifolds with sphere boundary}
\author{Bojun Zhao}
\address{Department of Mathematics, University at Buffalo, Buffalo, NY 14260, USA}
\email{bojunzha@buffalo.edu}
\maketitle

\begin{abstract}
	Let $M$ be a closed orientable irreducible $3$-manifold
	such that $\pi_1(M)$ is left orderable.
	
	(a)
	Let $M_0 = M - Int(B^{3})$,
	where $B^{3}$ is a compact $3$-ball in $M$.
	We have a process to produce 
	a co-orientable Reebless foliation $\mathcal{F}$ in $M_0$
    such that:
    (1)
    $\mathcal{F}$ has a transverse $(\pi_1(M),\mathbb{R})$ structure,
    (2)
    there exists 
    a simple closed curve in $M$ that
    is co-orientably transverse to $\mathcal{F}$ and
    intersects every leaf of $\mathcal{F}$.
    More specifically,
    given a pair $(<,\Gamma)$ composed of a left-invariant order ``$<$'' of $\pi_1(M)$ 
    and a fundamental domain $\Gamma$ of $M$ in its universal cover with certain property
    (which always exists),
    we can produce a resulting foliation in $M - Int(B^{3})$ as above,
    and we can test if it can extend to a taut foliation of $M$.
    
    (b)
    Suppose further that $M$ is either atoroidal or a rational homology $3$-sphere.
    If $M$ admits an $\mathbb{R}$-covered foliation $\mathcal{F}_0$,
    then there is a resulting foliation $\mathcal{F}$ of our process in $M - Int(B^{3})$ such that:
    $\mathcal{F}$ can extend to 
    an $\mathbb{R}$-covered foliation $\mathcal{F}_{extend}$ of $M$,
    and
    $\mathcal{F}_0$ can be recovered from doing 
    a collapsing operation on
    $\mathcal{F}_{extend}$.
    Here,
    by a collapsing operation on $\mathcal{F}_{extend}$,
    we mean the following process:
    (1) choosing an embedded product space
    $S \times I$ in $M$ for
    some (possibly non-compact) surface $S$ such that
    $S \times \{0\},
    S \times \{1\}$ are leaves of $\mathcal{F}_{extend}$
    (notice that $\mathcal{F}_{extend} \mid_{S \times I}$ may not be 
    a product bundle),
    (2) replacing $\mathcal{F}_{extend} \mid_{S \times I}$ by a single leaf $S$.
    
    (c)
    We conjecture that there always exists
    a resulting foliation of our process in $M - Int(B^{3})$ which
    can extend to a taut foliation in $M$.
\end{abstract}

\section{Introduction}

Throughout this paper,
all $3$-manifolds are assumed to be orientable,
and all foliations and laminations
are assumed to be co-orientable.

The L-space conjecture was proposed by Boyer-Gordon-Watson 
in \hyperref[BGW]{[BGW]} and 
by Juh\'asz in \hyperref[J]{[J]}:

\begin{conjecture}[L-space conjecture]\label{L-space conjecture}
	Let $M$ be an orientable irreducible rational homology $3$-sphere.
	Then the following statements are equivalent:
	
	(1)
	$M$ is a non-L-space.
	
	(2)
	$\pi_1(M)$ is left orderable.
	
	(3)
	$M$ admits a co-orientable taut foliation.
\end{conjecture}

Conjecture \ref{L-space conjecture} has been proved for
graph manifolds in
\hyperref[BC]{[BC]}, \hyperref[Ra]{[Ra]}, \hyperref[HRRW]{[HRRW]}.
The implication (3) $\Rightarrow$ (1)
is known by
\hyperref[OS]{[OS]},
\hyperref[B]{[B]},
\hyperref[KR]{[KR]}.
In \hyperref[Ga1]{[Ga1]},
Gabai proved that
all closed orientable irreducible $3$-manifolds with 
positive first Betti number
admit taut foliations.
In \hyperref[BRW]{[BRW]},
Boyer, Rolfsen and Wiest proved that
for every compact orientable irreducible $3$-manifold $M$ with $b_1(M) > 0$,
$\pi_1(M)$ is left orderable.
Hence (3) and (2)
hold for all closed orientable irreducible $3$-manifolds 
with $b_1 > 0$.

Furthermore,
connections between 
foliations and laminations,
and orders and group actions (on $1$-manifolds)
were developed in many works.
Through his universal circle construction
(\hyperref[T]{[T]}),
Thurston showed that:
for every atoroidal $3$-manifold $M$ that admits a taut foliation,
there exists an effective action of $\pi_1(M)$ on $S^{1}$
induced by the taut foliation.
In \hyperref[CD]{[CD]},
Calegari and Dunfield showed that
tight essential laminations
with solid tori guts in every atoroidal $3$-manifold $M$ 
also induce
effective actions of $\pi_1(M)$ on $S^{1}$.

Transverse $(\pi_1,\mathbb{R})$ structure for taut foliations is
a way to relate them to
left orderability of the fundamental group.

\begin{defn}\rm\label{transverse structure}
Suppose that $M$ is a closed orientable irreducible $3$-manifold
which admits a taut foliation $\mathcal{F}$.
Let $p: \widetilde{M} \to M$ be the universal covering of $M$,
and let $L$ be the leaf space of the pull-back foliation of $\mathcal{F}$ in
$\widetilde{M}$.
There is a natural action of $\pi_1(M)$ on $L$ induced by the deck transformations of $\widetilde{M}$,
called the \emph{$\pi_1$-action} on $L$.
A \emph{transverse $(\pi_1(M),\mathbb{R})$ structure} on $\mathcal{F}$ is
an immersion $i_{des}: L \to \mathbb{R}$ that
descends the $\pi_1$-action on $L$ to
a nontrivial action of $\pi_1(M)$ on $\mathbb{R}$
via homeomorphisms,
i.e. $i_{des}$ induces a homomorphism
$d: \pi_1(M) \to Homeo_+(\mathbb{R})$ such that the following diagram commutes (for every $g \in \pi_1(M)$):
\begin{center}	
	\begin{tikzcd}
		L \arrow[r, "i_{des}"] \arrow[d, "g"']
		& \mathbb{R} \arrow[d, "d(g)"] \\
		L \arrow[r, "i_{des}"]
		& \mathbb{R}
	\end{tikzcd}
\end{center}
Here, by an immersion,
we mean a topological immersion
(i.e. a local homeomorphism).
Furthermore,
$\mathcal{F}$ is called \emph{$\mathbb{R}$-covered} if $L \cong \mathbb{R}$.
\end{defn}

Notice that if $\mathcal{F}$ has
a transverse $(\pi_1(M),\mathbb{R})$ structure,
then $\pi_1(M)$ acts on $\mathbb{R}$ nontrivially,
and therefore
$\pi_1(M)$ is left orderable
(cf. \hyperref[BRW]{[BRW, Theorem 3.2]}).
The following question was proposed by Thurston
(cf. \hyperref[Cal3]{[Cal3, Question 8.1]}):

\begin{question}[Thurston]\label{Thurston}
	Let $M$ be a closed orientable irreducible $3$-manifold 
	such that $\pi_1(M)$ is left orderable.
	When does $M$ admit a taut folation
	with
	a transverse $(\pi_1(M),\mathbb{R})$ structure?
\end{question}

See \hyperref[T]{[T]} for 
some motivations for Question \ref{Thurston} and
\hyperref[Cal3]{[Cal3, Question 8.1, Remark]} for
more information on Question \ref{Thurston}.
It's clear that every $\mathbb{R}$-covered foliation in a closed $3$-manifold has
a natural transverse $(\pi_1,\mathbb{R})$ structure.
In \hyperref[Z]{[Z]},
Zung showed that a large class of closed $3$-manifolds admit
(non-$\mathbb{R}$-covered) taut foliations with a similar property.

At first,
we consider the related topics for
compact $3$-manifolds with $S^{2}$ boundary.
Our first result provides
foliations in this setting which are analogous to taut foliations in
closed $3$-manifolds:

\begin{thm}\label{theorem 1}
	Let $M$ be a closed orientable irreducible $3$-manifold
	such that $\pi_1(M)$ is left orderable.
	Let $M_0 = M - Int(B^{3})$,
	where $B^{3}$ is a compact $3$-ball in $M$.
	Then there is a Reebless foliation $\mathcal{F}$ in $M_0$,
	where the leaves of $\mathcal{F}$ may be
	transverse to $\partial M_0$
	or tangent to $\partial M_0$
	at their intersections with
	$\partial M_0$,
	such that:
	
	(1)
	$\mathcal{F}$ is analogous to ``taut'' in the following sense:
	
	$\bullet$
	There is a simple closed curve in $M$ 
	that is co-orientably transverse to $\mathcal{F}$ and
	has nonempty intersection with every leaf of $\mathcal{F}$.
	
	(2)
	$\mathcal{F}$ has
	a transverse $(\pi_1(M),\mathbb{R})$ structure:
	
	$\bullet$
    Let $L$ denote
	the leaf space of the pull-back foliation of $\mathcal{F}$
	in the universal cover of $M_0$,
	then there exists an immersion $i_{des}: L \to \mathbb{R}$
	descending the $\pi_1$-action on $L$
	to an effective action on $\mathbb{R}$.
	
	(3)
	Every transverse intersection component of some leaf of $\mathcal{F}$ and 
	$\partial M_0$ is a circle.
	
	Furthermore,
	we have a process to produce such foliations.
	$\mathcal{F}$ is called a \emph{resulting foliation} of this process.
\end{thm}

\begin{remark}\rm \label{remark}
	(a)
	There are many equivalent definitions for
	taut foliations in closed $3$-manifolds
	(cf. \hyperref[Cal4]{[Cal4, Chapter 4]}, \hyperref[Cal5]{[Cal5]}).
	The property of $\mathcal{F}$ in Theorem \ref{theorem 1} (1) is only analogous to one of these definitions
	for taut foliations:
	there exists a compact $1$-manifold 
	(i.e. a finite union of simple closed curves)
	transverse to the foliation that intersects every leaf.

	(b)
	Theorem \ref{theorem 1} (2) implies that
	$\pi_1(M)$ acts on $\mathbb{R}$ effectively.
	Hence the existence of the foliation $\mathcal{F}$
	as given in Theorem \ref{theorem 1}
	is a necessary and sufficient condition for
	$\pi_1(M)$ being left orderable.
	
	(c)
	$\mathcal{F}$ has the following property:
	if a simple closed curve in $M$ is co-orientably transverse to $\mathcal{F}$ and
	has nonempty intersection with some leaves of $\mathcal{F}$,
	then it is essential in $M$ (cf. Remark \ref{monotone}).
\end{remark}

It's clear that Theorem \ref{theorem 1} is a necessary condition for $M$ to
admit a taut foliation with transverse $(\pi_1(M), \mathbb{R})$ structure.
We suspect that it is also sufficient.
In the following,
we show that many resulting foliations of our process can extend to
$\mathbb{R}$-covered foliations of $M$.

\subsection{$\mathbb{R}$-covered foliations constructed in our process}
$\mathbb{R}$-covered foliations are studied extensively in atoroidal $3$-manifolds in many works,
which provide plentiful examples and 
various geometric and topological properties.
See \hyperref[Fe1]{[Fe1]}, \hyperref[T]{[T]}, \hyperref[Cal1]{[Cal1]},
\hyperref[Cal2]{[Cal2]}, \hyperref[Fe2]{[Fe2]} for example.
The following question is given in \hyperref[Cal3]{[Cal3, Question 8.3]}:

\begin{question}
	Given an atoroidal $3$-manifold that admits a taut foliation, must it admit $\mathbb{R}$-covered foliations?
\end{question}

The following result indicates that:
combined with an operation,
all $\mathbb{R}$-covered foliations in irreducible atoroidal $3$-manifolds and rational homology $3$-spheres
can be constructed by our process:

\begin{defn}\rm\label{collapsing operation}
	Let $\mathcal{F}_0$ be a co-orientable taut foliation of a closed $3$-manifold $M$.
	Suppose that there is an embedded product region $S \times I$ in $M$
	(where $S$ is a possibly non-compact surface) such that
	$S \times \{0\}, S \times \{1\}$ are leaves of $\mathcal{F}_0$.
	Notice that $\mathcal{F}_0 \mid_{S \times I}$ may not be a product bundle.
	We can replace the region $S \times I \subseteq M$ by a single surface $S$ and
	replace $\mathcal{F}_0 \mid_{S \times I}$ by the single leaf $S$.
	Then we obtain a foliation $\mathcal{F}^{'}_{0}$ of $M$ which
	is necessarily co-orientable and taut.
	We call the process $\mathcal{F}_0 \rightsquigarrow \mathcal{F}^{'}_{0}$ 
	a \emph{collapsing operation}.
\end{defn}

\begin{thm}\label{process}
	Let $M$ be either an orientable irreducible atoroidal $3$-manifold or
	an orientable rational homology $3$-sphere.
	Suppose that $M$ admits an $\mathbb{R}$-covered foliation $\mathcal{F}_0$.
	Then there is a resulting foliation $\mathcal{F}$ of $M - Int(B^{3})$ obtained from
	the process in Theorem \ref{theorem 1},
	such that:
	
	(1)
	$\mathcal{F}$ can extend to 
	an $\mathbb{R}$-covered foliation $\mathcal{F}_{extend}$ of $M$.
	
	(2)
	$\mathcal{F}_0$ can be recovered from doing a collapsing operation on $\mathcal{F}_{extend}$.
\end{thm}

\subsection{The test condition for extendability}\label{subsection 1.1}

We provide the input of our process,
and we provide
the condition to test if a resulting foliation of our process in $M - Int(B^{3})$ 
can extend to a taut foliation in $M$ through the input.

Let $M$ be a closed orientable irreducible $3$-manifold
such that $\pi_1(M)$ is left orderable.
Let $p: \widetilde{M} \to M$ be the universal covering of $M$.
We fix a base point $\widetilde{x} \in \widetilde{M}$.
Let $x = p(\widetilde{x})$,
and let $G = \pi_1(M,x)$.

\begin{defn}\rm\label{order-domain pair}
	A fundamental domain $\Gamma$ of $M$ in $\widetilde{M}$ is
	called \emph{standard} if 
	$p(\partial \Gamma)$ is a standard spine
	(cf. Definition \ref{standard spine}).
	An \emph{order-domain pair} $(<,\Gamma)$ of $M$ is 
	a pair of
	a left invariant order ``$<$'' of $G$
	and 
	a standard fundamental domain $\Gamma$ of $M$ in $\widetilde{M}$.
\end{defn}

In Lemma \ref{standard spine proof},
we show that $M$ always has 
a standard fundamental domain $\Gamma$ in $\widetilde{M}$,
and
$\Gamma$ can be constructed in finitely many steps.
To be convenient,
we assume $\widetilde{x} \in Int(\Gamma)$.
For each $h \in G$,
we denote by $t_h: \widetilde{M} \to \widetilde{M}$
the deck transformation such that:
if $\eta$ is a path in $\widetilde{M}$ that starts at $\widetilde{x}$
and ends at $t_h(\widetilde{x})$,
then $p(\eta)$ is a loop in $M$ with
$[p(\eta)] = h$.

A compact region $F \subseteq \partial \Gamma$ is called 
a \emph{face} of $\Gamma$ if
there is $h \in G - \{1\}$ such that
$F$ is a component of $\Gamma \cap t_h(\Gamma)$ and
$F \nsubseteq \Gamma \cap t_g(\Gamma)$ for arbitrary
$g \in G - \{1,h\}$.
And we call $F$ a \emph{positive face}
(resp. \emph{negative} face) if
$h > 1$ (resp. $h < 1$).

\begin{defn}\rm
	An order-domain pair $(<,\Gamma)$ of $M$ is called a \emph{good pair} if
	the union of positive faces of $\Gamma$ is 
	a single $2$-disk.
	And we call $(<,\Gamma)$ a \emph{very good pair} if
	(1)
	$(<,\Gamma)$ is good,
	(2)
	both $\bigcup_{g \in G, g < 1}t_g(\Gamma)$ and 
	$\bigcup_{g \in G, g \geqslant 1}t_g(\Gamma)$ are connected.
\end{defn}

The input of our process and the test condition for extendability are as follows:

\begin{thm}\label{extendable}
	Let $(<,\Gamma)$ be an arbitrary order-domain pair of $M$.
	
	(a)
	$(<,\Gamma)$ produces a resulting foliation through the process in Theorem \ref{theorem 1}
	(called a \emph{resulting foliation} of $(<,\Gamma)$),
	which is uniquely deteremined by $(<,\Gamma)$ up to blowing-up/down.
	
	(b)
	A resulting foliation of $(<,\Gamma)$ can extend to a taut foliation of $M$ if and only if
	$(<,\Gamma)$ is a good pair.
	Moreover,
	if it can extend to a taut foliation $\mathcal{F}_1$ of $M$,
	then $\mathcal{F}_1$ has a transverse $(\pi_1(M),\mathbb{R})$ structure.
	
	(c)
	A resulting foliation of $(<,\Gamma)$ can extend to an $\mathbb{R}$-covered foliation of $M$ if and only if
	$(<,\Gamma)$ is a very good pair.
\end{thm}

Combined with Theorem \ref{theorem 1} and Theorem \ref{process},
we have

\begin{cor}\label{summary}
    Suppose that $M$ is either an orientable irreducible atoroidal $3$-manifold or
    an orientable irreducible rational homology $3$-sphere.
    
    (a)
	$M$ admits an $\mathbb{R}$-covered foliation if and only if
	there is a very good pair $(<,\Gamma)$ of $M$.
	
	(b)
	We have a process with input all very good pairs $(<,\Gamma)$ of $M$ and
	output all $\mathbb{R}$-covered foliations of $M$,
	up to blowing-up/down, collapsing operation, and its inverse.
\end{cor}

Notice that the space of left-invariant orders of $\pi_1(M)$ is either finite or uncountable
(cf. \hyperref[Lin]{[Lin]}).

\begin{cor}
    Under the assumption of Corollary \ref{summary},
	if the space of left-invariant orders of $\pi_1(M)$ is finite,
	then $M$ admits countably many distinct $\mathbb{R}$-covered foliations,
	up to blowing-up/down, collapsing operation, and its inverse.
\end{cor}

Our result motivates the following conjecture:

\begin{conjecture}\label{conjecture}
	Fix an arbitrary left-invariant order ``$<$'' of $G$.
	There is a process to produce
	a standard fundamental domain $\Gamma$ of $M$ in $\widetilde{M}$
	such that
	$(<,\Gamma)$ is a good pair.
\end{conjecture}

\subsection{Organization}
In Section \ref{section 2},
we provide some basic settings of this paper and
review some preliminaries on branched surfaces.

For a $3$-manifold $M$ as in Theorem \ref{theorem 1},
we construct a branched surface $B$ in $M$ in Section \ref{section 3}.
In Subsections \ref{subsection 4.1}$\sim$\ref{subsection 4.4},
we use $B$ to
construct the foliation $\mathcal{F}$ as required in 
Theorem \ref{theorem 1}.
In Subsection \ref{subsection 4.6},
we prove Theorem \ref{extendable}.
In Subsection \ref{subsection 4.7},
we discuss the $2$-dimensional case.
We prove Theorem \ref{process} in Section \ref{section 5}.

\begin{postscript}\rm
After we posted the first version of our paper to arXiv,
Baik, Hensel and Wu provided a new method to prove
Theorem \ref{theorem 1} under 
the assumption that $M$ admits
a strongly essential $1$-vertex trangulation in \hyperref[BHW]{[BHW]}.
\end{postscript}

\subsection{Acknowledgements}
    The author wishes to thank Professor Xingru Zhang for much help and support
    from him during this work,
    and for many suggestions and discussions
    that benefit the author a lot.
    The auther is grateful to 
    Professor Tao Li for answering some questions helpfully.
    
\section{Preliminaries}\label{section 2}

\subsection{Conventions}\label{convention}

In this paper,
all actions on manifolds are assumed to be orientation-preserving.
All intervals in $\mathbb{R}$ are assumed to have
increasing orientation,
and all homeomorphisms between intervals in $\mathbb{R}$ are 
assumed to be orientation-preserving.
By an \emph{immersion} from
a (possibly non-Hausdorff) $1$-manifold to
$\mathbb{R}$,
we will always mean a topological immersion
(i.e. a local homeomorphism).
For a set $X$,
we denote by $|X|$ the cardinality of $X$.
For metric spaces $A$ and $B$,
we denote by $A \setminus \setminus B$
the closure of $A - B$ under the path metric.
And we regard two foliations in a $3$-manifold as the same one if they are same up to isotopy of the $3$-manifold.

For any compact oriented $3$-manifold $N$ with
nonempty boundary,
the \emph{positive orientation} of $\partial N$ with respect to
the orientation of $N$ is
assumed to be the orientation on $\partial N$ induced from
the orientation of $N$.

A \emph{transversal with endpoints} of a foliation is
a closed interval such that
its interior is transverse to the foliation.
By a \emph{transversal} of a foliation, 
we will always mean a transversal with endpoints.
A transversal of a co-oriented foliation is
\emph{positively oriented} (resp. \emph{negatively oriented}) if
its direction is consistent with (resp. opposite to) the
co-orientation on the foliation.

Throughout this paper,
we will always have the following notations:

\begin{notation}\rm\label{M}
	Let $M$ be a closed, orientable, irreducible $3$-manifold with left orderable fundamental group.
	Let $p: \widetilde{M} \to M$ be the universal covering of $M$.
	We choose an orientation on $M$,
	which induces an orientation on $\widetilde{M}$ such that $p$ is orientation preserving.
	
	(a)
	We fix a base point $\widetilde{x} \in \widetilde{M}$,
	and we denote $p(\widetilde{x})$ by $x$.
	Let $G = \pi_1(M,x)$.
	For each $h \in G$,
	we denote by $t_h: \widetilde{M} \to \widetilde{M}$ 
	the deck transformation of $\widetilde{M}$ such that:
	if $\eta \subseteq \widetilde{M}$ is a path that 
	starts at $\widetilde{x}$ and ends at $t_h(\widetilde{x})$,
	then $[p(\eta)] = h$.
	
	(b)
	We fix a left-invariant order ``$<$'' of $G$.
\end{notation}

\subsection{Branched surfaces}

In this section,
we review some background materials on branched surfaces.
The branched surface is an important tool to
describe taut foliations (cf. \hyperref[Ga2]{[Ga2]}) and
essential laminations (cf. \hyperref[GO]{[GO]} and \hyperref[Li]{[Li]}).
Our notations and definitions of branched surfaces
follow from
\hyperref[FO]{[FO]},
\hyperref[O]{[O]},
\hyperref[Ga2]{[Ga2]},
\hyperref[Li]{[Li]}:

\begin{figure}\label{standard spine picture}
	\centering
	\includegraphics[width=0.3\textwidth]{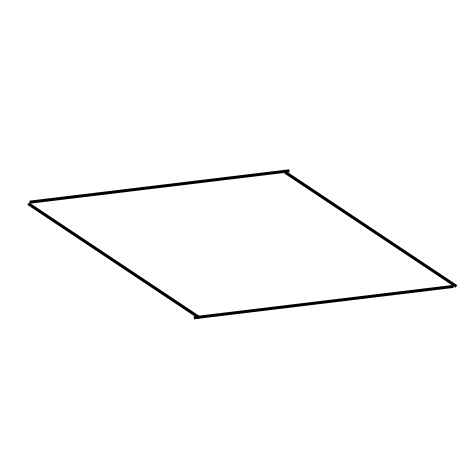}
	\includegraphics[width=0.3\textwidth]{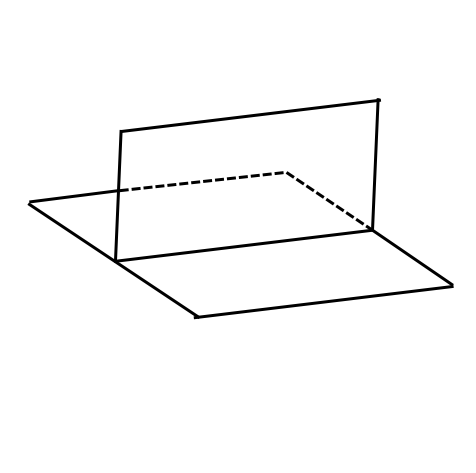}
	\includegraphics[width=0.3\textwidth]{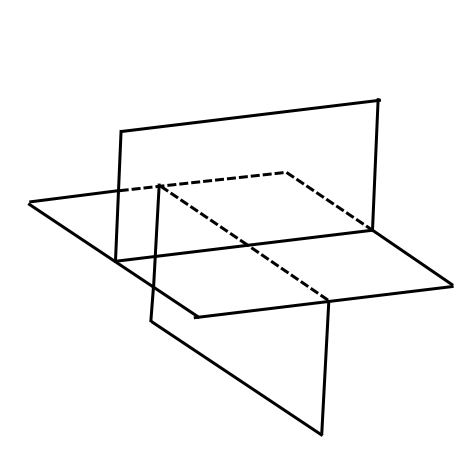}
	\caption{The local models of
		standard spines (\hyperref[Cas]{[Cas]}).}
\end{figure}

\begin{figure}\label{branched surface}
	\centering
	\includegraphics[width=0.3\textwidth]{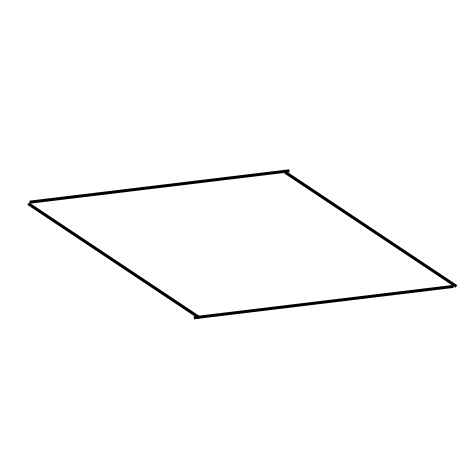}
	\includegraphics[width=0.3\textwidth]{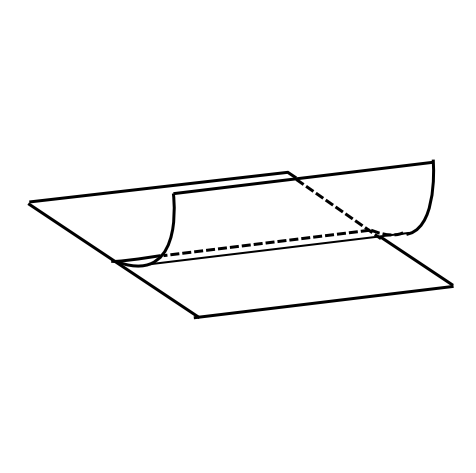}
	\includegraphics[width=0.3\textwidth]{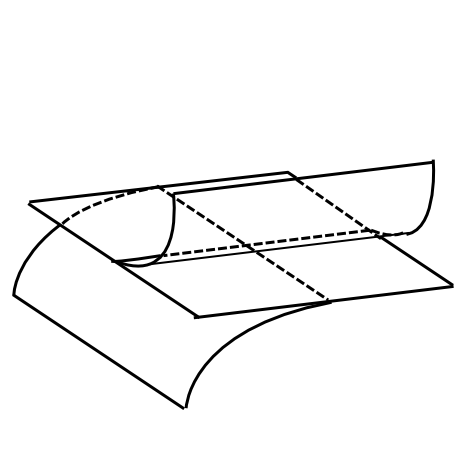}
	\caption{The local models of
		branched surfaces.}
\end{figure}

\begin{defn}[Standard spine]\rm \label{standard spine}
	(a)
	A standard spine (\hyperref[Cas]{[Cas]}) 
	is a $2$-complex
	locally modeled as Figure \ref{standard spine picture}.
	
	(b)
	Let $S$ be a standard spine.
	Let $L$ be the set of points in $S$
	that have no Euclidean neighborhood in $S$.
	We call each component of $S \setminus \setminus L$
	a \emph{sector} of $S$.
	And we call each point in $L$ that 
	has no $\mathbb{R}$-neighborhood in $L$
	a \emph{double point}.
\end{defn}

\begin{defn}\rm \label{branched surface definition}
	A \emph{branched surface} $B$ in $M$ is 
	an embedded standard spine with 
	well-defined cusp directions,
	which is
	locally modeled as Figure \ref{branched surface}.
\end{defn}

\begin{notation}\rm\label{branched surface notation}
Let $B$ be a branched surface in $M$.
	
(a)
Let $N(B)$ denote a regular neighborhood of the branched surface $B$
as shown in Figrue \ref{fibered neighborhood}.
Then $N(B)$ can be regarded as 
an interval bundle over $B$,
and we call these fibers
\emph{$I$-fibers} (or \emph{interval fibers}).
$N(B)$ is called
a \emph{fibered neighborhood} of $B$.

(b)
$\partial N(B)$ can be decomposed to 
$\partial_h N(B)$ (\emph{horizontal boundary})
and $\partial_v N(B)$ (\emph{vertical boundary}) such that:
$\partial_h N(B)$ is transverse to $I$-fibers,
and $\partial_v N(B)$ is tangent to $I$-fibers.

(c)
Let $\pi: N(B) \to B$ 
be the projection that
takes each $I$-fiber to a single point.
We call $\pi$ the \emph{collapsing map} for $N(B)$.

(d)
A lamination $\mathcal{L}$ 
is \emph{carried} by $B$
if there exists a fibered neighborhood $N(B)$
such that 
$\mathcal{L} \subseteq N(B)$ and
each leaf of $\mathcal{L}$ is transverse to 
the $I$-fibers of $N(B)$.
$\mathcal{L}$ 
is \emph{fully carried} by $B$
if it is carried by $B$ and intersects
every $I$-fiber of $N(B)$.

(e)
The \emph{branch locus} of $B$
is the set of points in $B$
that have no Euclidean neighborhood in $B$.
We denote it by $L(B)$.
Each component of $B \setminus \setminus L(B)$ 
is called a \emph{branch sector}.
We call each point in $L(B)$
that has no $\mathbb{R}$-neighborhood in $L(B)$
a \emph{double point} of $L(B)$.
Let $V$ denote the set of double points of $L(B)$.
We call each component of 
$L(B) \setminus \setminus V$
an \emph{edge} of $L(B)$.
\end{notation}

\begin{figure}\label{fibered neighborhood}
	\includegraphics[width=0.7\textwidth]{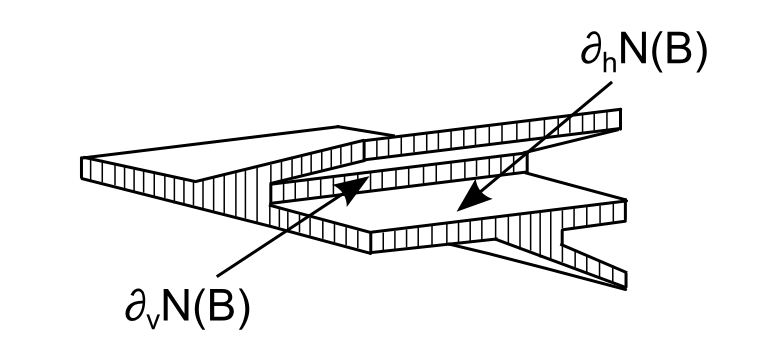}
	\caption{The fibered neighborhood $N(B)$.}
\end{figure}

\begin{remark}\rm\label{pull-back branched surface}
	Suppose that $B$ is a branched surface in $M$.
	Let
	$\widetilde{B} = p^{-1}(B) \subseteq \widetilde{M}$.
	
	(a)
	By convention,
	we also say that
	$\widetilde{B}$ is a \emph{branched surface}
	in $\widetilde{M}$.
	
	(b)
	We'd like to have the fibered neighborhood of $\widetilde{B}$ be
	$\pi_1$-equivariant.
	We call $\widetilde{N(B)}$ a \emph{fibered neithborhood} of
	$\widetilde{B}$ if
	there is a fibered neighborhood $N(B)$ of $B$
	such that
	$\widetilde{N(B)} = p^{-1}(N(B))$,
	and we will always assume that 
	the $I$-fibers of $\widetilde{N(B)}$
	are the pull-back of the $I$-fibers of $N(B)$.
	Let
	$\widetilde{\pi}: \widetilde{N(B)} \to \widetilde{B}$ be
	the map which collapses every $I$-fiber of
	$\widetilde{N(B)}$ into a point of $\widetilde{B}$.
	We also call $\widetilde{\pi}$ the \emph{collapsing map} for
	$\widetilde{N(B)}$.
	And we adopt the notations in
	Notation \ref{branched surface notation} (b), (d), (e) for $\widetilde{B}$.
\end{remark}

\subsection{Blowing-up/down}
Blowing-up/down a leaf of a foliation describes the following operation:

\begin{defn}[Blowing-up/down]\rm\label{blow-up}
	\emph{Blowing-up} a leaf $\lambda$ of a foliation is to 
	replace it by a product 
	$\lambda \times I$
	foliated with leaves $\{\lambda \times \{t\} \mid t \in I\}$,
	and its inverse operation is called
	\emph{blowing-down}
	(cf. \hyperref[Cal4]{[Cal4, Example 4.14]}).
	Such a bundle $\lambda \times I$ is called 
	a \emph{pocket} of the foliation.
	The blowing-up operation can be done for
	a countable collection of disjoint leaves simultaneously,
	and the blowing-down operation can be done for
	a countable collection of disjoint pockets simultaneously.
\end{defn}

\section{Constructing a branched surface $B$ of $M$}\label{section 3}

In this section,
we construct a branched surface $B$ in $M$ through the following two steps:
(1)
choosing a fundamental domain $\Gamma$ of $M$ in $\widetilde{M}$ such that
$p(\partial \Gamma)$ is a standard spine
(cf. Subsection \ref{subsection 3.1}),
(2)
orienting the sectors of $p(\partial \Gamma)$ through ``$<$'' to
obtain a branched surface $B$ in $M$
(cf. Subsection \ref{subsection 3.2}).

\subsection{Choosing a fundamental domain 
$\Gamma$ of $M$ in $\widetilde{M}$}\label{subsection 3.1}

Our setting for fundamental domains of $M$ in $\widetilde{M}$ is as follows:

\begin{notation}\rm\label{domain}
	(a)
	We call $\Gamma$ a \emph{fundamental domain} of $M$ in $\widetilde{M}$
	if $\Gamma \subseteq \widetilde{M}$ is 
	a connected, compact, contractible polyhedron
	with topologically locally flat boundary,
	$p \mid_{Int(\Gamma)}$ is injective,
	$p \mid_{\Gamma}$ is surjective,
	and $\partial \Gamma$
	is a finite union of polygons
	that are identified pairwise by deck transformations.
	
	(b)
	A compact region $F \subseteq \partial \Gamma$
	is called a \emph{face} of $\Gamma$
	if there is
	$h \in G - \{1\}$ such that 
	$F$ is a component of
	$\Gamma \cap t_h(\Gamma)$,
	and there is no $g \in G - \{1,h\}$
	such that 
	$F \subseteq \Gamma \cap t_g(\Gamma)$.
	$F$ is also called 
	a \emph{common face} of
	$\Gamma$ and $t_h(\Gamma)$.
\end{notation}

\begin{fact}\rm
	$\widetilde{M}$ is homeomorphic to $\mathbb{R}^{3}$,
	and every fundamental domain of $M$ in $\widetilde{M}$ is homeomorphic to a compact $3$-ball.
\end{fact}
\begin{proof}
	$G$ is torsion-free since it is left orderable
	(cf. \hyperref[BRW]{[BRW, Proposition 2.1]}).
	The fact follows.
\end{proof}

\begin{fact}\rm
	Let $\Gamma$ be a fundamental domain of $M$ in $\widetilde{M}$.
	Then every compact set in $\widetilde{M}$ only meets
	finitely many members in $\{t_g(\Gamma) \mid g \in G\}$.
\end{fact}
\begin{proof}
	This follows from \hyperref[Gr]{[Gr, Lemma 1]}.
\end{proof}

\begin{notation}\rm
	Let $\Gamma$ be a fundamental domain of $M$ in $\widetilde{M}$.
	We denote by $G(\Gamma)$
	the set of points in $p(\partial \Gamma)$
	that have no 
	$\mathbb{R}^{2}$-neighborhood in $p(\partial \Gamma)$.
	A point in $G(\Gamma)$ is called a \emph{vertex} of $G(\Gamma)$ if
	it has no $\mathbb{R}^{1}$-neighborhood in $G(\Gamma)$.
	And we call every component of 
	$G(\Gamma) \setminus \setminus \{$vertices of $G(\Gamma)\}$ 
	an \emph{edge} of $G(\Gamma)$.
	An edge $e$ of $G(\Gamma)$ is called $k$-\emph{valent} if
	$p^{-1}(e) \cap \partial \Gamma$ has $k$ components.
	A vertex $v$ of $G(\Gamma)$ is called \emph{standard} if
	$v$ has a neighborhood in $p(\partial \Gamma)$ which is
	homeomorphic to 
	the local model of a double point of a standard spine
	(cf. the rightmost picture of Figure \ref{standard spine picture}).
	And we call every component of $p(\partial \Gamma) \setminus \setminus G(\Gamma)$
	a \emph{sector of $p(\partial \Gamma)$}.
\end{notation}

\begin{lm}\label{standard spine proof}
	There exists a fundamental domain $\Gamma$ of $M$ in $\widetilde{M}$
	such that $p(\partial \Gamma)$
	is a standard spine.
\end{lm}

\begin{proof}
	Let $\Gamma_0$
	be a fundamental domain of $M$ in $\widetilde{M}$.
	Starting from $\Gamma_0$,
	we repeat the following operation 
	to obtain $\Gamma_{i+1}$ from $\Gamma_i$ ($i \geqslant 0$):
	
	$\bullet$
	(Thicken an edge)
	Assume that $\Gamma_i$ is 
	a fundamental domain of $M$ in $\widetilde{M}$,
	and there is an edge $e$ of $G(\Gamma_i)$
    which is not $3$-valent.
	Let $N_e = e \times B^{2}$ be 
	a sufficiently small embedded disk bundle over $e$ in $M$ such that
	$e = N_e \cap G(\Gamma_i)$.
	Let $\widetilde{e}$ be a component of
	$p^{-1}(e)$ such that
	$\widetilde{e} \subseteq \partial \Gamma_i$.
	Let $\widetilde{N_e}$ be the component of $p^{-1}(N_e)$
	such that $\widetilde{e} \subseteq \widetilde{N_e}$.
	Let 
	$$\Gamma_{i+1} = \overline{\Gamma_i - p^{-1}(N_e)} \cup \widetilde{N_e}$$
	(see Figure \ref{thicken an edge} for the change).
	Then $\Gamma_{i+1}$ is still 
	a fundamental domain of $M$ in $\widetilde{M}$,
	and every new edge of $G(\Gamma_{i+1})$ produced in this step
	has valency $3$.
	Thus,
	\begin{center}
		$|\{s$ is an edge of $G(\Gamma_{i+1})$, $s$ is not $3$-valent$\}|
		< 
		|\{s$ is an edge of $G(\Gamma_i)$, $s$ is not $3$-valent$\}|$.
	\end{center}

\begin{figure}\label{thicken an edge}
	\centering
	\subfigure[]{
	\includegraphics[width=0.4\textwidth]{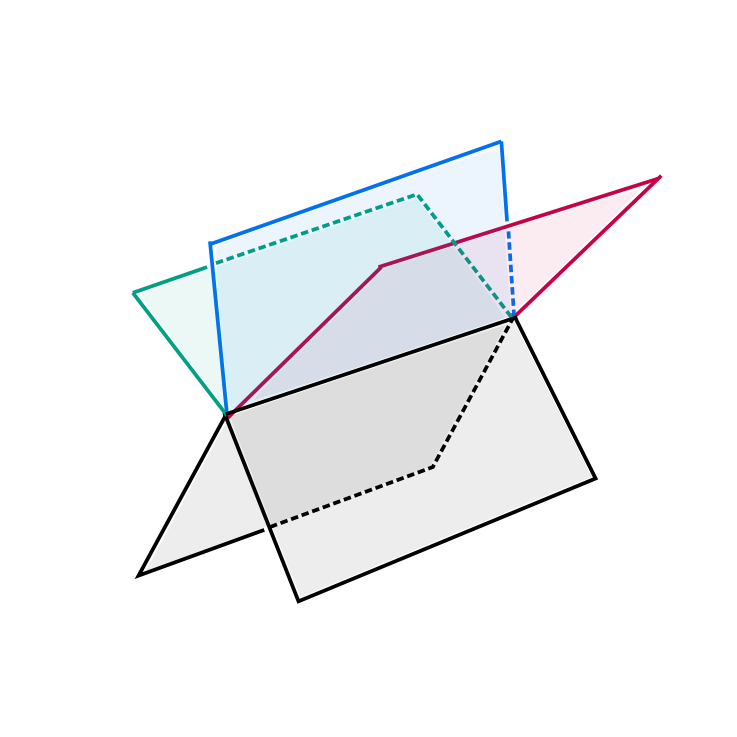}}
    \subfigure[]{
	\includegraphics[width=0.4\textwidth]{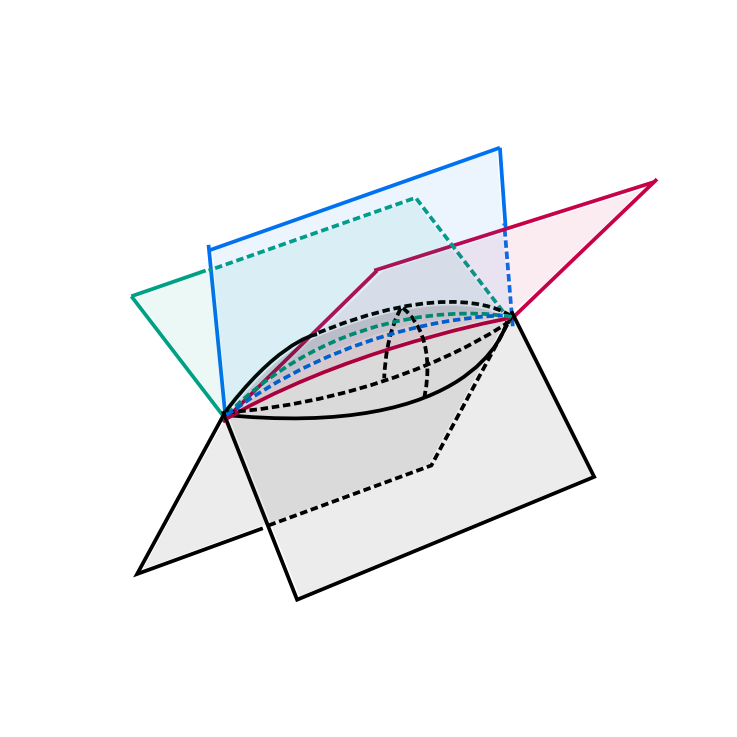}}
	\caption{(a) is the picture near
		$\widetilde{e}$ before the thickening,
		and (b) is the picture near
		$\widetilde{e}$ after the thickening.
		Here,
		the two grey faces in (a) are contained in $\partial \Gamma_i$,
		and the grey faces in (b) are contained in $\partial \Gamma_{i+1}$.}
\end{figure}
	
	We can repeat the above operation inductively,
	until we obtain $\Gamma_k$ ($k \in \mathbb{N}$)
	such that 
	all edges of $G(\Gamma_k)$ have valency $3$.
	Starting from $\Gamma_k$,
	we repeat the following operation
	to obtain $\Gamma_{i+1}$ from $\Gamma_i$ ($i \geqslant k$):
	
	$\bullet$
	(Thicken a vertex)
	Assume that $\Gamma_i$ is 
	a fundamental domain of $M$ in $\widetilde{M}$
	such that every edge of $G(\Gamma_i)$ is $3$-valent,
	and there is a vertex $v$ of $G(\Gamma_i)$
	which is not standard.
	Let $N_v$ be a sufficiently small closed regular neighborhood of
	$v$ in $M$ so that:
	$N_v \cap p(\partial \Gamma_i)$ is 
	a closed neighborhood of $v$ in $p(\partial \Gamma_i)$ that
	does not contain any other vertices of $G(\Gamma_i)$,
	and $N_v \cap G(\Gamma_i)$ is connected.
	Let $\widetilde{v} \in p^{-1}(v)$ such that
	$\widetilde{v} \in \partial \Gamma_i$.
	Let $\widetilde{N_v}$ be the component of $p^{-1}(N_v)$
	such that $\widetilde{v} \in \widetilde{N_v}$.
	Let 
	$$\Gamma_{i+1} = \overline{\Gamma_i - p^{-1}(N_v)} \cup \widetilde{N_v}$$
	(see Figure \ref{thicken a vertex} for the change).
	Then $\Gamma_{i+1}$ is still 
	a fundamental domain of $M$ in $\widetilde{M}$.
	Each vertex of $G(\Gamma_{i+1})$ that lies on $\partial N_v$ is
	the intersection of $\partial N_v$ and 
	an edge of $G(\Gamma_i)$ (which is necessarily $3$-valent),
	so it is standard.
	And each edge of $G(\Gamma_{i+1})$ that lies on $\partial N_v$
	is the intersection of $\partial N_v$ and a sector of $p(\partial \Gamma_i)$,
	so it has valency $3$.
	Hence
	\begin{center}
		$|\{$vertices of $G(\Gamma_{i+1})$ 
		which are not standard$\}|
		< 
		|\{$vertices of $G(\Gamma_i)$
		which are not standard$\}|$,
	\end{center}
	and
	all new edges produced in this process still have valency $3$.
	
	\begin{figure}\label{thicken a vertex}
		\centering
		\subfigure[]{
		\includegraphics[width=0.4\textwidth]{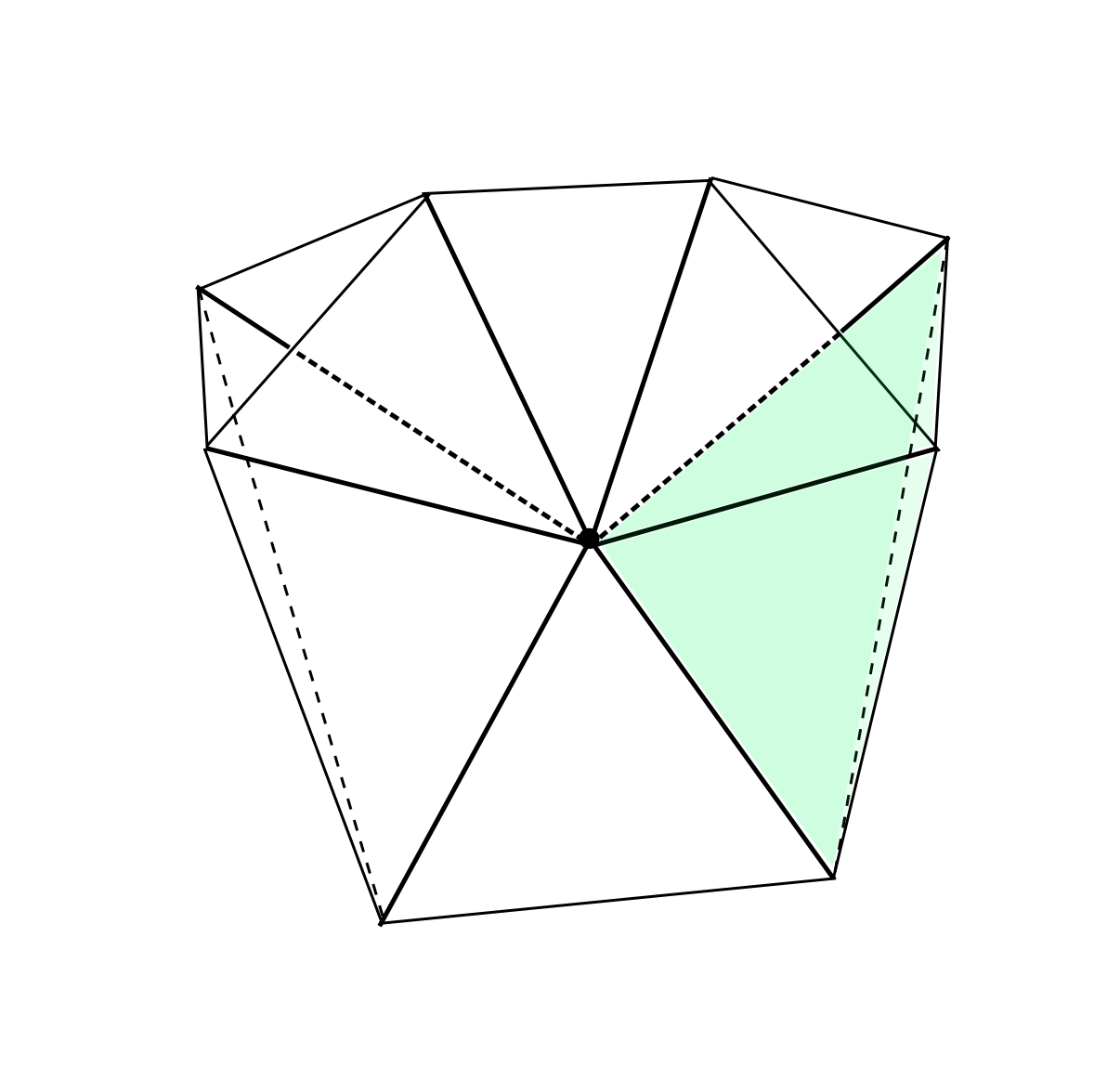}}
		\subfigure[]{
		\includegraphics[width=0.4\textwidth]{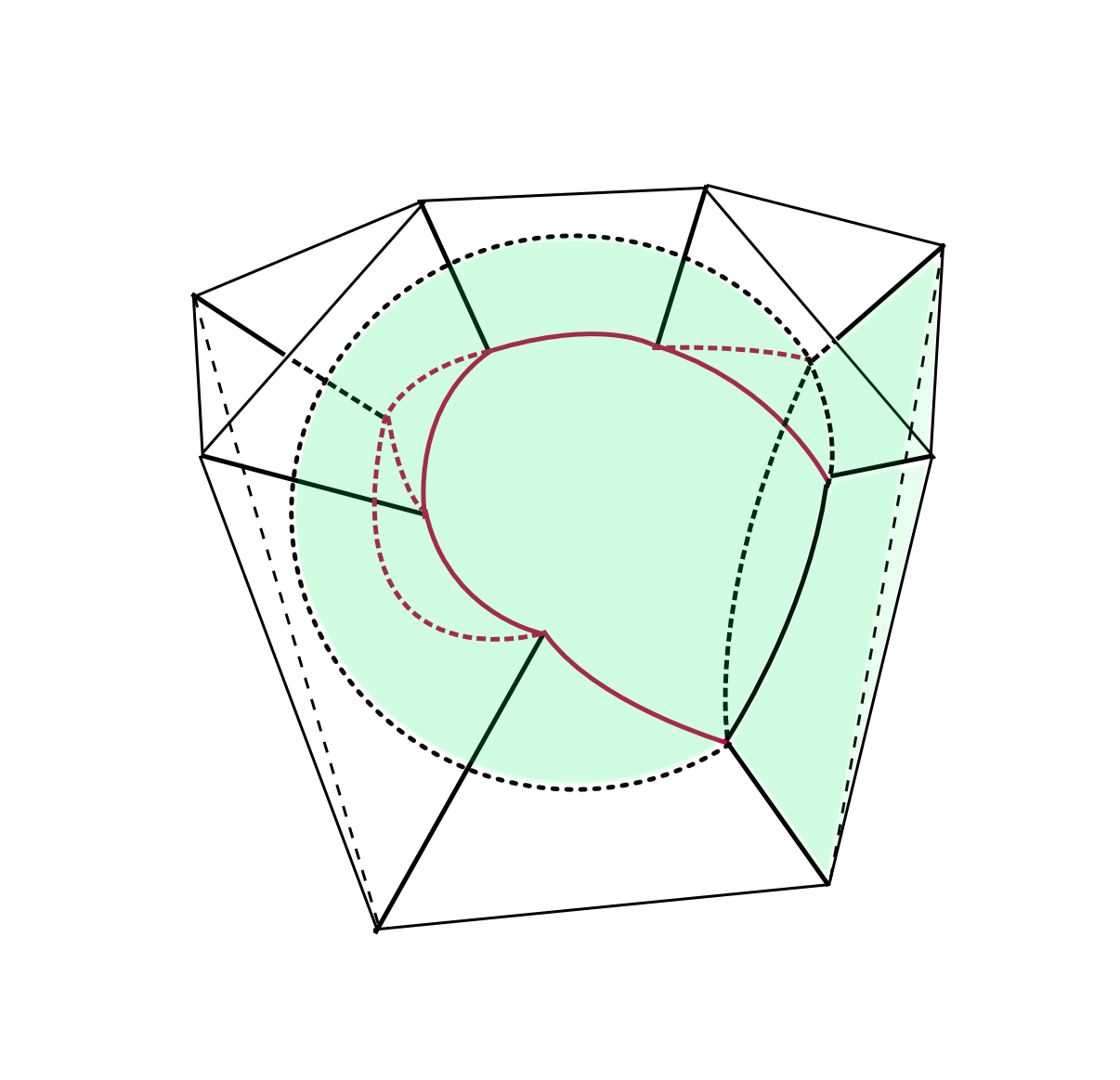}}
		\caption{(a) is the picture near $\widetilde{v}$ before the thickening,
			and (b) is the picture near $\widetilde{v}$ after the thickening.
			Here,
			the green faces in (a) are contained in $\partial \Gamma_i$,
			the green faces in (b) are contained in $\partial \Gamma_{i+1}$,
		    and the red curves in (b) are projected to the edges of $G(\Gamma_{i+1})$ which
	        are produced or changed in this step.}
	\end{figure}
	
	We can repeat this operation inductively.
	At last,
	we can obtain a fundamental domain $\Gamma_n$ such that
	all vertices of $G(\Gamma_n)$ are standard,
	and all edges of $G(\Gamma_n)$
	are still $3$-valent.
	Then
	$p(\partial \Gamma_n)$ is a standard spine.
	We can accomplish the proof of
	Lemma \ref{standard spine proof}
	by choosing
	$\Gamma = \Gamma_n$.
\end{proof}

Guaranteed by Lemma \ref{standard spine proof},
we can choose a fundamental domain $\Gamma$ of $M$ in $\widetilde{M}$ 
such that $p(\partial \Gamma)$ is a standard spine.
Without loss of generality,
we assume that the base point $\widetilde{x}$ (cf. Notation \ref{M} (a)) is contained in 
$Int(\Gamma)$.
And we assume that every member of
$\{t_g(\Gamma) \mid g \in G\}$ has an orientation induced from
the orientation on $\widetilde{M}$.

\subsection{Orienting the sectors of $p(\partial \Gamma)$ to obtain
a branched surface of $M$}\label{subsection 3.2}

Recall from Notation \ref{domain} (b),
a common face of $\Gamma$ and $t_h(\Gamma)$ ($h \in G - \{1\}$) is
a component of
$\Gamma \cap t_h(\Gamma)$ such that there is no $g \in G - \{1,h\}$ with
$F \subseteq \Gamma \cap t_g(\Gamma)$.
 
\begin{defn}\rm\label{face}
	Suppose that $F$ is 
	a common face of $\Gamma$ and $t_h(\Gamma)$ for some
	$h \in G - \{1\}$.
	We call $F$
	a \emph{positive face} (resp. \emph{negative face}) of $\Gamma$ if 
	$h > 1$ (resp. $h < 1$).
	For every $r \in G$ and every face $J$ of 
	the fundamental domain $t_r(\Gamma)$,
	there is a face $K$ of $\Gamma$ such that 
	$t_r(K) = J$.
	We call $J$ a \emph{positive face} (resp. \emph{negative face}) of $t_r(\Gamma)$ if 
	$K$ is 
	a positive face (resp. negative face) of $\Gamma$.
\end{defn}

\begin{remark}\rm
	The covering map $p: \widetilde{M} \to M$ takes
	$\{\text{faces of } \Gamma\}$ to
	$\{\text{sectors of } p(\partial \Gamma)\}$.
	In particular,
	$\{\text{positive faces of } \Gamma\}$ is in
	one-to-one correspondence with
	$\{\text{sectors of } p(\partial \Gamma)\}$.
\end{remark}

For every face in every member of $\{t_g(\Gamma)\}$,
``$<$'' determines if it is positive/negative:

\begin{lm}\label{orientation}
	Suppose that $F$ is a common face of two fundamental domains
	$t_h(\Gamma)$ and $t_g(\Gamma)$
	($h,g \in G$ and $h \ne g$).
	Without loss of generality,
	we assume that $h > g$.
	Then $F$ is a positive face of 
	$t_g(\Gamma)$
	and is a negative face of $t_h(\Gamma)$.
\end{lm}
\begin{proof}
	$t^{-1}_{g}(F)$ is a common face of
	$\Gamma$ and $t_{g^{-1}h}(\Gamma)$.
    We have $g^{-1}h > 1$ since $h > g$.
    Then $t^{-1}_{g}(F)$ is a positive face of $\Gamma$,
    which implies that
    $F$ is a positive face of $t_g(\Gamma)$.
    
    Similarly,
    $t^{-1}_{h}(F)$ is a common face of
    $\Gamma$ and $t_{h^{-1}g}(\Gamma)$, 
    and $h^{-1}g < 1$.
    So $t^{-1}_{h}(F)$ is a negative face of $\Gamma$,
    and then
    $F$ is a negative face of $t_h(\Gamma)$.
\end{proof}

Then we orient the faces of $\{t_g(\Gamma)\}$ in 
a $\pi_1$-equivariant way,
which induces a co-orientation on the sectors of $p(\partial \Gamma)$:

\begin{defn}\label{orienting}\rm
	(a)
	For each $h \in G$ and each face $F$ of $t_h(\Gamma)$,
	we assign $F$ positive orientation 
	(resp. negative orientation) with respect to 
	the orientation on $t_h(\Gamma)$ if
	$F$ is a positive face 
	(resp. negative face) of $t_h(\Gamma)$.
	By Lemma \ref{orientation},
	every face is a positive face of a member of 
	$\{t_g(\Gamma)\}_{g \in G}$ that
	contains it,
	and it is a negative face of the other member of 
	$\{t_g(\Gamma)\}_{g \in G}$ that contains it.
	Hence the assignment is well-defined.
	
	(b)
	By Definition \ref{face},
	the assignment in (a) is equivariant under deck transformations.
	So it induces
	a co-orientation on the sectors of the standard spine
	$p(\partial \Gamma)$.
	We assign $p(\partial \Gamma)$ the cusp directions
	given by this co-orientation.
\end{defn}

At last, 
we verify that 
$p(\partial \Gamma)$
is a branched surface.

\begin{lm}\label{tangent space}
	$p(\partial \Gamma)$ is a branched surface.
\end{lm}

\begin{proof}
	As shown in Figure \ref{orient},
	the orientations on faces of 
	$\{t_g(\Gamma)\}_{g \in G}$ give
	well-defined cusp directions to the set of points in
	$\bigcup_{g \in G}t_g(\partial \Gamma)$
	that possess no Euclidean neighborhood.
	So $\bigcup_{g \in G}t_g(\partial \Gamma)$ 
	is locally modeled as
	Figure \ref{branched surface}.
	Correspondently,
	$p(\partial \Gamma)$
	is locally modeled as
	Figure \ref{branched surface},
	and therefore
	$p(\partial \Gamma)$ is a branched surface.
\end{proof}

\begin{figure}\label{orient}
	\centering
	\subfigure[]{
	\includegraphics[width=0.35\textwidth]{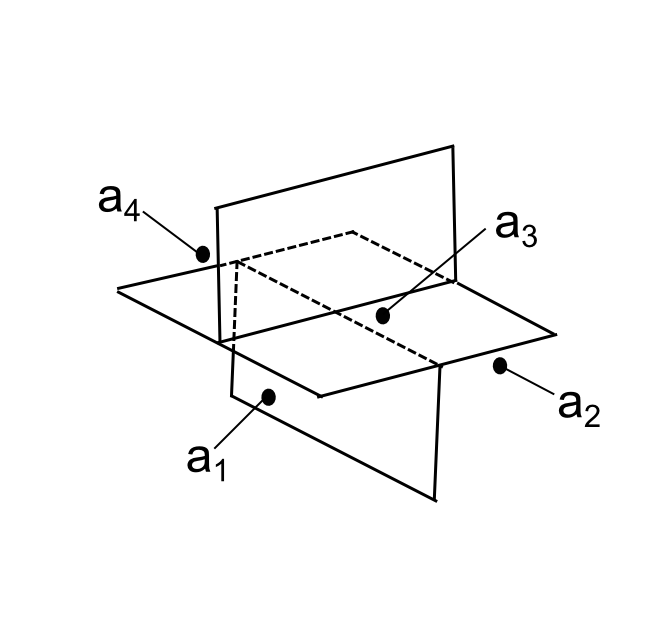}}
    \subfigure[]{
	\includegraphics[width=0.35\textwidth]{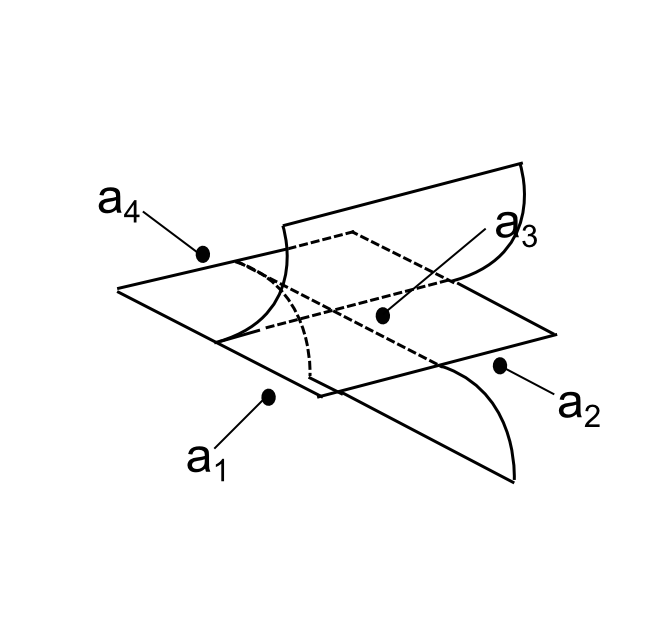}}
	\caption{Assume that the four points
	$a_1,a_2,a_3,a_4 \in \widetilde{M}$ 
	in (a)
	lies in
	$t_{h_1}(\Gamma)$,
	$t_{h_2}(\Gamma)$,
	$t_{h_3}(\Gamma)$,
	$t_{h_4}(\Gamma)$
	respectively and
	$h_1 < h_2 < h_3 < h_4$.
	Then (a) is oriented to (b).}
\end{figure}

We denote this branched surface by $B = p(\partial \Gamma)$.
We choose a fibered neighborhood 
$N(B)$ of $B$.
Let $\widetilde{B} = p^{-1}(B)$.
As in Remark \ref{pull-back branched surface},
$\widetilde{B}$ is a branched surface in $\widetilde{M}$,
and $\widetilde{B}$ has 
a fibered neighborhood $\widetilde{N(B)} = p^{-1}(N(B))$ with
$\pi_1$-equivariant interval fibers.
We denote by 
$\widetilde{\pi}: \widetilde{N(B)} \to \widetilde{B}$
the collapsing map for $\widetilde{N(B)}$.

\section{The proof of Theorem \ref{theorem 1}}\label{section 4}

In this section,
we follow all assumptions, settings and notations from last section,
including
the fundamental domain $\Gamma$ of $M$ in $\widetilde{M}$ and
the branched surface $B$ in $M$.

\subsection{The action of $G$ on $\mathbb{R}$}\label{subsection 4.1}

To begin with,
we describe the blowing-up of an action,
which follows from 
\hyperref[D]{[D]} and
\hyperref[Cal4]{[Cal4, Construction 2.45]}.

\begin{defn}[Denjoy]\rm\label{Denjoy}
	Suppose that
	$\{r_g: \mathbb{R} \to \mathbb{R} \mid g \in G\}$ is
	an effective action of $G$ on $\mathbb{R}$ such that
	a point $v \in \mathbb{R}$ has trivial stabilizer.
	We blow-up the countable union of points 
	$\{r_g(v)\}_{g \in G}$ in $\mathbb{R}$,
	i.e. replacing every point $r_g(v)$ ($g \in G$) by 
	a blown-up interval $r_g(v) \times I$.
	We consider the complement of 
	$\bigcup_{g \in G}(r_g(v) \times I)$ in $\mathbb{R}$
	to be the same set as
	$\mathbb{R} - \bigcup_{g \in G} \{r_g(v)\}$.
	For each $h \in G$,
	let $s_h: \mathbb{R} \to \mathbb{R}$ be the map such that
	(1)
	the restriction of
	$s_h$ to $\mathbb{R} - \bigcup_{g \in G} (r_g(v) \times I)$ is equal to
	the restriction of $r_h$ to $\mathbb{R} - \bigcup_{g \in G} \{r_g(v)\}$,
	(2)
	$s_h$ takes $r_g(v) \times \{t\}$ to
	$r_{hg}(v) \times \{t\}$
	for all $g \in G$, $t \in I$.
	Then $\{s_g: \mathbb{R} \to \mathbb{R} \mid g \in G\}$ is 
	an action of $G$ on $\mathbb{R}$.
	We call it the \emph{blowing-up} of
	$\{r_g: \mathbb{R} \to \mathbb{R} \mid g \in G\}$ at $v$.
\end{defn}

Under the conditions of Definition \ref{Denjoy},
we will always assume $r_g(v) \times \{0\} < r_g(v) \times \{1\}$ for every $g \in G$.
In the following,
we define an action of $G$ on $\mathbb{R}$ which is compatible with
the left-invariant order ``$<$'' on $G$:

\begin{prop}\label{rho^'}
	There exists an effective action
	$\{\rho^{'}_{g}: \mathbb{R} \to \mathbb{R} \mid g \in G\}$ of $G$ on
	$\mathbb{R}$ and 
	$N \in \mathbb{R}$ such that:
	for arbitrary $g,h \in G$,
	$\rho^{'}_{h}(N) > \rho^{'}_{g}(N)$ if and only if $h > g$.
\end{prop}
\begin{proof}
	This follows from \hyperref[Cal4]{[Cal4, Lemma 2.43]} and
	\hyperref[Cal4]{[Cal4, Lemma 2.43, Remark]}.
	Precisely,
	through the process in \hyperref[Cal4]{[Cal4, Lemma 2.43]},
	there is an effective action of $G$ on $\mathbb{R}$ with 
	the following additional properties:
	
	$\bullet$
    There is an injective map
    $e: G \to \mathbb{R}$ such that
    $e(g) < e(h)$ for any $g,h \in G$ with $g < h$.
    
    $\bullet$
    For any $g, h \in G$,
    the action
    $g: \mathbb{R} \to \mathbb{R}$ transforms $e(h)$ to $e(gh)$.
    
    We choose $\{\rho^{'}_{g}: \mathbb{R} \to \mathbb{R} \mid g \in G\}$
    to be this action and choose
    $N = e(1)$. 
    Then the properties of the proposition are satisfied.
\end{proof}

\begin{defn}\rm\label{rho}
	Let $\{\rho_g: \mathbb{R} \to \mathbb{R} \mid g \in G\}$ be the
	blowing-up of 
	$\{\rho^{'}_{g}: \mathbb{R} \to \mathbb{R} \mid g \in G\}$ at $N$.
	We denote by $N \times I$ the blown-up interval at $N$.
	Let $N_0 = N \times \{0\}$,
	$N_1 = N \times \{1\}$.
\end{defn}

Note that
(1) $N_1 > N_0$,
(2) $N_0,N_1$ have trivial stabilizer under the action
$\{\rho_g: \mathbb{R} \to \mathbb{R} \mid g \in G\}$,
(3) for arbitrary $g,h \in G$ with $h > g$ and arbitrary $i,j \in \{0,1\}$,
$\rho_h(N_i) > \rho_g(N_j)$.

\begin{defn}\rm\label{I_F}
(a)
Let
$Sec(\widetilde{B}) = 
\{$branch sectors of $\widetilde{B}\}$.

(b)
Let $F \in Sec(\widetilde{B})$.
We denote by $s,r \in G$ for which
$F$ is a common face of $t_s(\Gamma)$, $t_r(\Gamma)$ and
$s < r$.
Let
$I_F = [\rho_s(N_1), \rho_r(N_0)]$.
\end{defn}

\begin{fact}\rm\label{psi-2}
	Let $F \in Sec(\widetilde{B})$, 
	$g \in G$.
	Then $I_{t_g(F)} = \rho_g(I_F)$.
\end{fact}
\begin{proof}
	We denote by $s,r \in G$ such that
	$F$ is a common face of $t_s(\Gamma)$, $t_r(\Gamma)$ and
	$s < r$.
	Then $t_g(F)$ is a common face of
	$t_{gs}(\Gamma)$, $t_{gr}(\Gamma)$,
	and $gr > gs$.
	By Definition \ref{I_F} (b), 
	we have
	$$I_{t_g(F)} = [\rho_{gs}(N_1),\rho_{gr}(N_0)] =
	\rho_g([\rho_s(N_1), \rho_r(N_0)]) = \rho_g(I_F).$$
\end{proof}

Recall from Subsection \ref{subsection 3.2},
a positive face of $\Gamma$ is a common face of
$\Gamma$ and $t_h(\Gamma)$ for some $h > 1$.

\begin{defn}\rm \label{psi}
	(a)
	For each positive face $F_0$ of $\Gamma$,
	we choose an orientation-preserving homeomorphism
	$$h_{F_0}: I \to I_{F_0}.$$
	
	(b)
	For each $F \in Sec(\widetilde{B})$,
	there is
	$r \in G$ and a positive face $F_0$ of $\Gamma$
	such that
	$F = t_r(F_0)$.
	Let
	$$h_F = \rho_r \circ h_{F_0}:
	I \to I_F.$$
\end{defn}

\begin{fact}\rm \label{transform}
	Let $F \in Sec(\widetilde{B})$,
	$g \in G$.
	Then
	$h_{t_g(F)} = \rho_g \circ h_F$.
\end{fact}

\begin{proof}
	There exists 
	$r \in G$ and a positive face $F_0$ of $\Gamma$
	such that
	$F = t_r(F_0)$.
	Then
	$$h_{t_g(F)} = h_{t_{gr}(F_0)} =
	\rho_{gr} \circ h_{F_0} =
	\rho_g \circ \rho_r \circ h_{F_0} =
	\rho_g \circ h_F.$$
\end{proof}

\begin{lm}\label{edge relation}
	Let $\textbf{e}$ be an edge of the branch locus of 
	$\widetilde{B}$.
	Let $F,X,Y \in Sec(\widetilde{B})$ be
	the three distinct branch sectors of $\widetilde{B}$ such that 
	$\textbf{e}$ is contained in $F,X,Y$
	and
	the cusp direction at $\textbf{e}$ points out of $X,Y$
	and points in $F$.
	We denote by 
	$u,v \in G$ such that 
	(1) $F$ is a common face between
	$t_u(\Gamma),t_v(\Gamma)$,
	(2)
	$Y \subseteq t_u(\Gamma)$,
	$X \subseteq t_v(\Gamma)$
	(cf. Figure \ref{transformation} (a)).
	And we assume without loss of generality that 
	$u > v$.
	Then:
	
	(a)
	$I_X,I_Y \subseteq I_F$
	and
	$I_X \cap I_Y = \emptyset$.
	
	(b)
	$\min I_F = \min I_X$,
	$\max I_F = \max I_Y$.
\end{lm}
\begin{proof}
	There exists exactly one
	$r \in G$ such that $X,Y \subseteq t_r(\Gamma)$
	(then $\textbf{e} \subseteq t_r(\Gamma)$, 
	and 
	the cusp direction at $\textbf{e}$ points out of $t_r(\Gamma)$,
	see Figure \ref{transformation} (a)).
	Then $u > r > v$.
	By Definition \ref{I_F} (b),
	we have
	$I_F = [\rho_v(N_1),\rho_u(N_0)]$,
	$I_X = [\rho_v(N_1),\rho_r(N_0)]$,
	$I_Y = [\rho_r(N_1),\rho_u(N_0)]$.
	Thus
	$$\min I_F = \min I_X < \max I_X < \min I_Y < \max I_Y = \max I_F,$$
	So (a), (b) holds.
\end{proof}

\begin{figure}\label{transformation}
	\centering
	\subfigure[]{
	\includegraphics[width=0.4\textwidth]{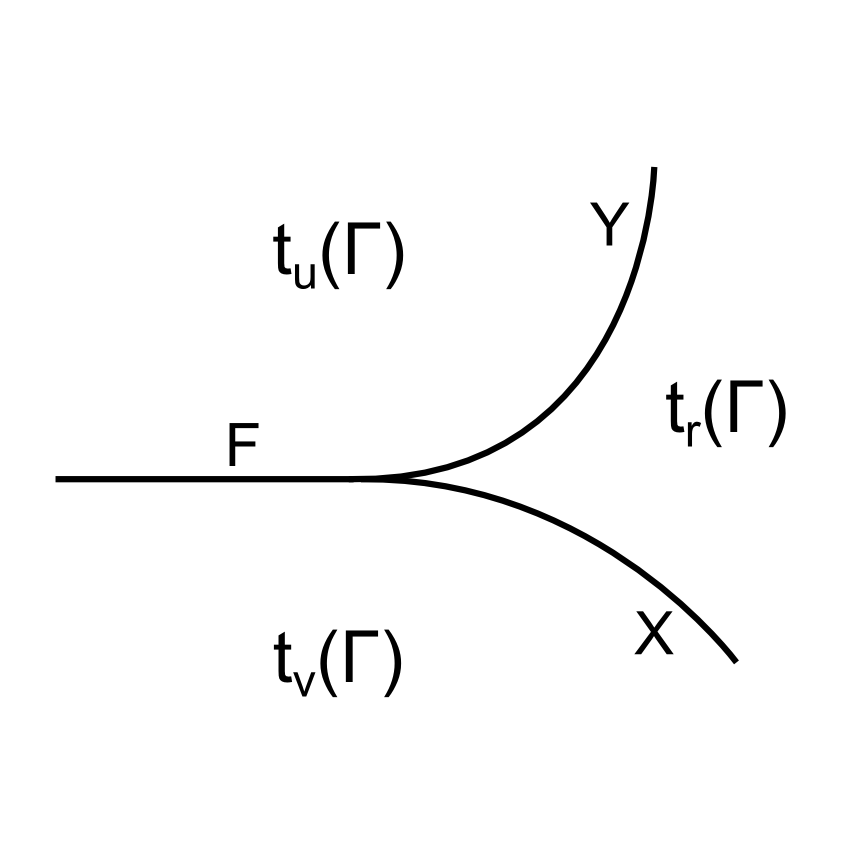}}
    \subfigure[]{
	\includegraphics[width=0.4\textwidth]{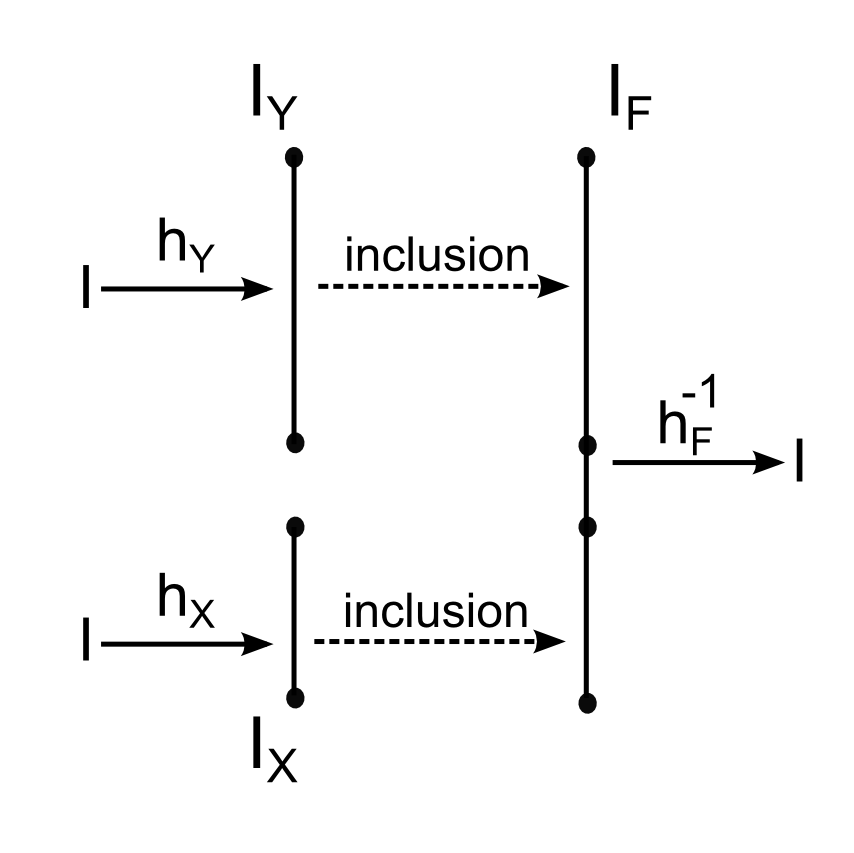}}
	\caption{(a) The picture of the branch sectors $X,Y,F$ of $\widetilde{B}$ and
		the members $t_u(\Gamma), t_v(\Gamma), t_r(\Gamma)$ in $\{t_g(\Gamma)\}_{g \in G}$,
		as given in Lemma \ref{edge relation}.
		(b) The inclusions of $I_X, I_Y$ into $I_F$,
		which induce the maps
		$\delta_{(X,F)}: I \stackrel{h_X}{\longrightarrow} 
		I_X \stackrel{\text{inclusion}}{\longrightarrow} 
		I_F \stackrel{h^{-1}_{F}}{\longrightarrow} I$, 
		$\delta_{(Y,F)}: I \stackrel{h_Y}{\longrightarrow} 
		I_Y \stackrel{\text{inclusion}}{\longrightarrow} 
		I_F \stackrel{h^{-1}_{F}}{\longrightarrow} I$.}
\end{figure}

\begin{defn}\rm \label{delta}
	Assume that the conditions of Lemma \ref{edge relation} hold.
	Then $I_X,I_Y \subseteq I_F$.
	Let 
	$$\delta_{(X,F)} = h^{-1}_{F} \circ h_X: I \to I,$$
	$$\delta_{(Y,F)} = h^{-1}_{F} \circ h_Y: I \to I.$$
\end{defn}

\begin{lm}
	Assume that the conditions of Lemma \ref{edge relation} hold.
	Then $\delta_{(X,F)}(I) \cap \delta_{(Y,F)}(I) = \emptyset$,
	$\delta_{(X,F)}(0) = 0$,
	$\delta_{(Y,F)}(1) = 1$.
\end{lm}
\begin{proof}
Follows from Lemma \ref{transform} directly
(cf. Figure \ref{transformation} (b)).
\end{proof}

Next,
we show that the maps
defined in Definition \ref{delta} are
$\pi_1$-equivariant:

\begin{lm}\label{equivariant gluing}
	Assume that the conditions of Lemma \ref{edge relation} hold.
	For each $g \in G$,
	we have
	$\delta_{(X,F)} = 
	\delta_{(t_g(X),t_g(F))}$ and
	$\delta_{(Y,F)} = 
	\delta_{(t_g(Y),t_g(F))}$.
\end{lm}

\begin{proof}
	We only show $\delta_{(X,F)} = 
	\delta_{(t_g(X),t_g(F))}$.
	By Fact \ref{transform}, 
	we have
	$$\delta_{(t_g(X),t_g(F))} = h^{-1}_{t_g(F)} \circ h_{t_g(X)} =
	(\rho_g \circ h_F)^{-1} \circ (\rho_g \circ h_X) = 
	h^{-1}_{F} \circ h_X = \delta_{(X,F)}.$$
\end{proof}

\subsection{Construction of the leaves}\label{subsection 4.4}

We have a decomposition of
$\widetilde{N(B)}$ to
$\coprod_{F \in Sec(\widetilde{B})} (F \times I)$
by cutting along
the union of interval fibers at the branch locus of $\widetilde{B}$.
We call each $F \times I$ 
($F \in Sec(\widetilde{B})$) a
\emph{product face} and call
each $F \times \{t\}$
($F \in Sec(\widetilde{B})$, $t \in I$)
a \emph{horizontal sheet}.
We assign a gluing to
$\coprod_{F \in Sec(\widetilde{B})} (F \times I)$ as follows:

\begin{construction}\rm\label{glue}
	For any edge $\textbf{e}$ of the branch locus of 
	$\widetilde{B}$ and
	branch sectors $F,X,Y$ of $\widetilde{B}$ which satisfy the conditions in
	Lemma \ref{edge relation},
	we glue $X \times I$, $Y \times I$ into
	$F \times I$ as follows:
	
	$\bullet$
	For our convenience,
	we only show the gluing between
	$X \times I$ and $F \times I$.
	Notice that there are two copies of $\textbf{e}$ in
	$F,X$ respectively.
	We denote by $\textbf{e}_X$ the copy of
	$\textbf{e}$ in $X$ and
	denote by $\textbf{e}_F$ the copy of
	$\textbf{e}$ in $F$
	(then $\textbf{e}_X \times I, \textbf{e}_F \times I$ are faces in $X \times I, F \times I$,
	respectively).
	Let
	$$\overline{\delta_{(X,F)}}: 
	\textbf{e}_X \times I \to \textbf{e}_F \times I$$
	be the map defined by
	$\overline{\delta_{(X,F)}}(r,t) = 
	(r,\delta_{(X,F)}(t))$
	for all
	$r \in \textbf{e}, t \in I$.
	We glue 
	$\textbf{e}_X \times I$
	into 
	$\textbf{e}_F \times I$
	through the map 
	$\overline{\delta_{(X,F)}}$ 
	(cf. Figure \ref{identification map}).
	As a result,
	for every $t \in I$,
	the segment
	$\textbf{e}_X \times \{t\} \subseteq \textbf{e}_X \times I$
	is homeomorphically glued with the segment
	$\textbf{e}_F \times \{\delta_{(X,F)}(t)\} \subseteq 
	\textbf{e}_F \times I$.
\end{construction}

\begin{figure}\label{identification map}
	\includegraphics[width=0.9\textwidth]{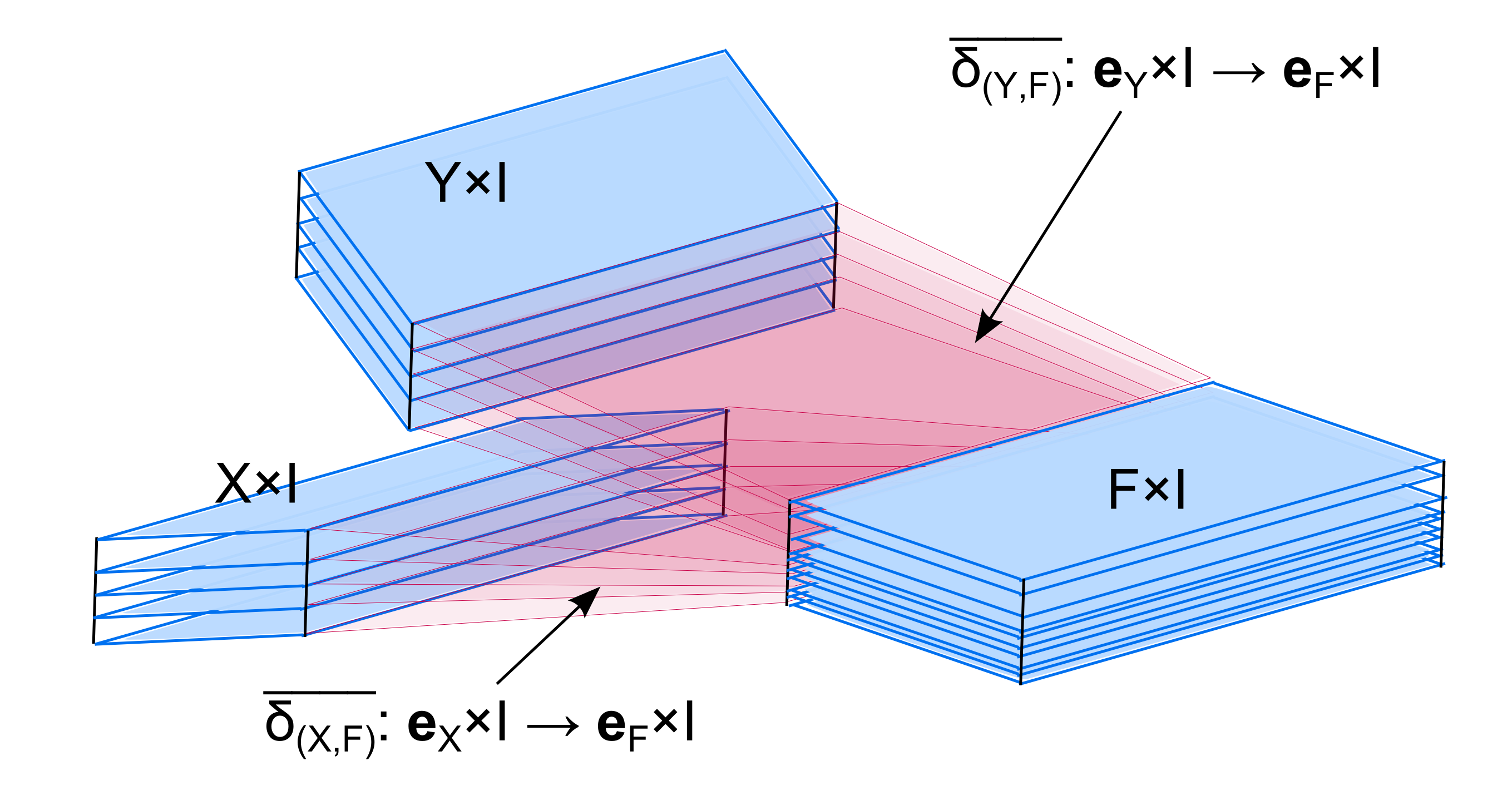}
	\caption{The gluing between
		$X \times I$, $Y \times I$ and
		$F \times I$.
		We glue $\textbf{e}_X \times I$, $\textbf{e}_Y \times I$ into
		$\textbf{e}_F \times I$
		and make
		$\{X \times \{t\}\}_{t \in I}$,
		$\{Y \times \{t\}\}_{t \in I}$
		attach to
		$\{F \times \{t\}\}_{t \in I}$.}
\end{figure}

\begin{defn}\rm\label{piece value}
	Let $$i_0: \coprod_{F \in Sec(\widetilde{B})} (F \times I) \to
	\mathbb{R}$$
	be the map such that:
	for each point $a \in F \times \{t\}$
	($F \in Sec(\widetilde{B}), t \in I$),
	we have
	$i_0(a) = h_F(t)$.
\end{defn}

\begin{fact}\rm\label{invariant}
	If two distinct points in
	$\coprod_{F \in Sec(\widetilde{B})} (F \times I)$ 
	are glued together by
	Construction \ref{glue},
	then their images under $i_0$ are equal.
\end{fact}
\begin{proof}
	Under the condition of Construction \ref{glue},
	we suppose that a segment $\textbf{e}_X \times \{u\} \subseteq
	X \times I$ ($u \in I$) is glued with
	a segment $\textbf{e}_F \times \{v\} \subseteq
	F \times I$ ($v \in I$).
	Then $v = \delta_{(X,F)}(u)$.
	By Definition \ref{delta},
	we have
	$h_X(u) = h_F(\delta_{(X,F)}(u)) = h_F(v)$.
	It follows that $i_0$ takes all points in 
	$\textbf{e}_X \times \{u\}$ and
	$\textbf{e}_F \times \{v\}$ to the same value in $\mathbb{R}$.
	Thus,
	any two points in
	$\coprod_{F \in Sec(\widetilde{B})} (F \times I)$ 
	are glued together only if they have the same image under $i_0$.
\end{proof}

Let $\mathcal{N}$ denote 
the resulting space obtained from gluing 
$\coprod_{F \in Sec(\widetilde{B})} (F \times I)$ as in
Construction \ref{glue}.
Notice that
the collection of horizontal sheets 
$\{F \times \{t\} \mid F \in Sec(\widetilde{B}), t \in I\}$ 
are attached edge-by-edge in the gluing.
So Construction \ref{glue} also induces
a gluing for
$\{F \times \{t\} \mid F \in Sec(\widetilde{B}), t \in I\}$.
Let $\mathcal{E}$ be
the resulting space obtained from the gluing for
$\{F \times \{t\} \mid F \in Sec(\widetilde{B}), t \in I\}$.
We have:

\begin{fact}\rm
	$\mathcal{N}$ is homeomorphic to $\widetilde{N(B)}$,
	and
	$\mathcal{E}$ is a foliation of $\mathcal{N}$.
\end{fact}
\begin{proof}
	Notice that the gluing
	$\overline{\delta_{(X,F)}}:
	\textbf{e}_X \times I \to \textbf{e}_F \times I$ 
	as given in Construction \ref{glue}
	homeomorphically embeds
	$\textbf{e}_X \times I$ into
	$\textbf{e}_F \times I$,
	and
	it glues the horizontal sheets edge-by-edge.
	Next,
	we focus on the gluings along the interval fibers at the
    double points of the branch locus of $\widetilde{B}$.
	
	Let $q$ be a double point of the branch locus of $\widetilde{B}$.
	We illustrate that
	Construction \ref{glue} gives compatible gluings between
	distinct copies of $\{q\} \times I$ in
	$\coprod_{F \in Sec(\widetilde{B})} (F \times I)$.
	Let $X,Y \in Sec(\widetilde{B})$ such that $q \in X, Y$.
	There are two distinct copies of $\{q\} \times I$ in
	$X \times I$, $Y \times I$ respectively.
	We denote them by 
	$\{q_X\} \times I \subseteq X \times I$ and
	$\{q_Y\} \times I \subseteq Y \times I$.
    By Fact \ref{invariant},
    if $a \in \{q_X\} \times I$,
    $b \in \{q_Y\} \times I$ are glued together,
    then $i_0(a)$ must be equal to $i_0(b)$.
    Notice that $i_0 \mid_{\{q_X\} \times I}, i_0 \mid_{\{q_Y\} \times I}$ are embeddings from
    $\{q_X\} \times I, \{q_Y\} \times I$ to $\mathbb{R}$.
    So the gluing between 
	$\{q_X\} \times I$ and
	$\{q_Y\} \times I$ is a homeomorphism between 
	their subintervals,
	and this gluing is compatible with other gluings between
	distinct copies of $\{q\} \times I$ in
	$\coprod_{F \in Sec(\widetilde{B})} (F \times I)$.
	
	Thus,
	for every point $t$ in an interval fiber at
	a double point,
	the horizontal sheets containing $t$ are glued horizontally at $t$.
	So
	$\mathcal{N}$ is homeomorphic to $\widetilde{N(B)}$,
	and
	$\mathcal{E}$ is a foliation of $\mathcal{N}$.
\end{proof}

We can regard $\mathcal{N}$ as $\widetilde{N(B)}$.
Combining Lemma \ref{equivariant gluing} and
Construction \ref{glue},
the gluing maps for
$\{F \times \{t\} \mid F \in Sec(\widetilde{B}), t \in I\}$ are 
$\pi_1$-equivariant.
So $\mathcal{E}$ is 
a $\pi_1$-equivariant foliation of
$\widetilde{N(B)}$.
Let $M_0 = N(B) \cong M - Int(B^{3})$.
The covering map $p: \widetilde{M} \to M$ descends $\mathcal{E}$ to
a foliation of $M_0$.
We denote it by $\mathcal{F}$.
For concreteness,
we denote $\mathcal{E}$ by $\widetilde{\mathcal{F}}$.
In the following,
we will always 
adopt the terminologies $\widetilde{N(B)}, \widetilde{\mathcal{F}}$
instead of $\mathcal{N}, \mathcal{E}$.

Note that
the leaves of $\mathcal{F}$ are tangent to $Int(\partial_h N(B))$ and 
transverse to $Int(\partial_v N(B))$.
In the following, 
we show that
$\mathcal{F}$ satisfies Theorem \ref{theorem 1} (1), (2), (3).

\begin{lm}\label{taut}
	(a)
	$\mathcal{F}$ contains no 
	torus leaf that bounds a solid torus
	(then $\mathcal{F}$ is Reebless).
	
	(b)
	There is a simple closed curve in $M$ that is co-orientably transverse to $\mathcal{F}$ and
	intersects every leaf of $\mathcal{F}$.
\end{lm}
\begin{proof}[The proof of (a)]
	Notice that every leaf of $\mathcal{F}$ is 
	carried by $B$.
	Since $B$ does not separate $M$,
	$B$ carries no torus leaf that
	bounds a solid torus.
\end{proof}
\begin{proof}[The proof of (b)]
	Let $q \in \Gamma - \widetilde{N(B)}$.
	We can choose a path $\tau$ in $\widetilde{M}$ such that:
	(1)
	$\tau$ starts at $q$,
	(2)
	for each face $F$ of $\Gamma$,
	$\tau$ meets at least one member of $\{t_g(F)\}_{g \in G}$,
	(3)
	if $\tau$ meets some $F \in Sec(\widetilde{B})$,
	then $\tau$ intersects $F$ transversely,
	and the direction of $\tau$ is consistent with the orientation on $F$
	(as a branch sector of $\widetilde{B}$),
	(4)
	$\tau$ ends at some point in $\{t_g(q)\}_{g \in G - \{1\}}$.
	We can isotope $\tau$ 
	(relative to its endpoints) so that
	$\tau$ 
	intersects $\widetilde{\mathcal{F}}$ transversely and
	$p(\tau)$ is a simple closed curve in $M$.
	Then $p(\tau)$ is the curve as required in (b).
\end{proof}

Next,
we show that 
$\mathcal{F}$ has a transverse $(\pi_1(M_0),\mathbb{R})$ structure.
Let $L$ be the leaf space of
$\widetilde{\mathcal{F}}$.
In the following,
we will not distinguish any point of 
$L$ with its corresponding leaf in
$\widetilde{\mathcal{F}}$.

Recall that we define the map
$i_0: \coprod_{F \in Sec(\widetilde{B})} (F \times I) \to
\mathbb{R}$ in Definition \ref{piece value}.

\begin{defn}\rm\label{strictly monotone map}
	Let $i_{des}: L \to \mathbb{R}$ be the map such that:
	for each $\widetilde{\lambda} \in L$,
	we choose an arbitrary point $v \in \widetilde{\lambda}$ and define
	$i_{des}(\widetilde{\lambda}) = i_0(v)$.
\end{defn}

By Fact \ref{invariant},
$i_{des}: L \to \mathbb{R}$ is well-defined.
And it's clear that $i_{des}$ is an immersion.

Next,
we show that $i_{des}$ descends the $\pi_1$-action on $L$ to
the action
$\{\rho_g: \mathbb{R} \to \mathbb{R} \mid g \in G\}$
(cf. Definition \ref{rho}).
Recall that the 
$\pi_1$-action on $L$ is induced from the 
deck transformations on $\widetilde{M}$.
For each $g \in G$, 
we still denote by
$t_g: L \to L$ the transformation on $L$
induced from the deck transformation
$t_g: \widetilde{M} \to \widetilde{M}$,
which
takes every leaf $\widetilde{\lambda} \in L$ to
$t_g(\widetilde{\lambda}) \in L$.

\begin{lm}\label{immersion}
	$i_{des}$ descends
	$\{t_g: L \to L \mid g \in G\}$ to
	$\{\rho_g: \mathbb{R} \to \mathbb{R} \mid g \in G\}$.
\end{lm}
\begin{proof}
	Let $\widetilde{\lambda} \in L$,
	$g \in G$.
	Suppose that $X \times \{s\}$ ($X \in Sec(\widetilde{B})$, $s \in I$) is a horizontal sheet contained in
	$\widetilde{\lambda}$.
	By Definition \ref{strictly monotone map} and
	Definition \ref{piece value},
	we have 
	$i_{des}(\widetilde{\lambda}) = 
	i_0(X \times \{s\}) = h_X(s)$.
	Combined with Fact \ref{transform},
	we have
	$$i_{des}(t_g(\widetilde{\lambda})) = 
	i_0(t_g(X) \times \{s\}) =
	h_{t_g(X)}(s) = \rho_g(h_X(s)) =
	\rho_g(i_{des}(\widetilde{\lambda})).$$
    It follows that
    $i_{des} \circ t_g = \rho_g \circ i_{des}$,
    i.e. the following diagram commutes:
    \begin{center}	
    	\begin{tikzcd}
    	L \arrow[r, "i_{des}"] \arrow[d, "t_g"']
    	& \mathbb{R} \arrow[d, "\rho_g"] \\
    	L \arrow[r, "i_{des}"]
    	& \mathbb{R}
    	\end{tikzcd}
    \end{center}
\end{proof}

We have shown that
$\mathcal{F}$ satisfies Theorem \ref{theorem 1} (1), (2).
It only remains to show
Theorem \ref{theorem 1} (3).

Let $V$ be a component of the vertical boundary of 
$\widetilde{N(B)}$
(then $V$ is an annulus).
Suppose that $\widetilde{\lambda}$ is a leaf of 
$\widetilde{\mathcal{F}}$ that meets $V$.
Notice that $i_{des}$ projects every transversal (with endpoints) of 
$\widetilde{\mathcal{F}}$ to a closed interval in $\mathbb{R}$.
Thus the two endpoints of 
any transversal of $\widetilde{\mathcal{F}}$ must be
contained in distinct leaves of $\widetilde{\mathcal{F}}$.
If $\widetilde{\lambda} \cap V$ is not a circle,
then there must be a transversal of $\widetilde{\mathcal{F}}$ in $V$ such that
its two endpoints are both contained in $\widetilde{\lambda} \cap V$.
This is a contradiction.
So $\widetilde{\lambda} \cap V$ is a circle.
Hence every transverse intersection component of some leaf of $\mathcal{F}$ and $\partial M_0$ is a circle.
So Theorem \ref{theorem 1} (3) holds. 

At last,
we explain that:
if a simple closed curve $\tau$ in $M$ is co-orientably transverse to $\mathcal{F}$ and
has nonempty intersection with some leaves of $\mathcal{F}$,
then $\tau$ is essential in $M$.

\begin{remark}\label{monotone}\rm
	Let $\tau: I \to M$ be a loop in $M$ such that: 
	(1)
	$\tau$ is positively transverse to $\mathcal{F}$,
	(2)
	$\tau(I)$ has nonempty intersection with
	some leaves of $\mathcal{F}$.
	Let $\widetilde{\tau}: I \to \widetilde{M}$ be a lift of $\tau$ in $\widetilde{M}$.
	Then $\widetilde{\tau}(I)$ meets at least two distinct leaves of $\widetilde{\mathcal{F}}$.
	We show $i_0(\widetilde{\tau}(1)) > i_0(\widetilde{\tau}(0))$ in the following.
	
	It's clear that $i_0(\widetilde{\tau}(t_1)) < i_0(\widetilde{\tau}(t_2))$ if
	$\widetilde{\tau}([t_1,t_2]) \subseteq \widetilde{N(B)}$
	(then $\widetilde{\tau}([t_1,t_2])$ is a positively oriented transversal of $\widetilde{\mathcal{F}}$).
	Now assume that $[t_1,t_2]$ is a subinterval of $I$ with
	$\widetilde{\tau}([t_1,t_2]) \cap \widetilde{N(B)} = \{\widetilde{\tau}(t_1),\widetilde{\tau}(t_2)\}$.
	We prove $i_0(\widetilde{\tau}(t_1)) < i_0(\widetilde{\tau}(t_2))$ as follows.
	Without loss of generality,
	we can assume $\widetilde{\tau}([t_1,t_2]) \subseteq \Gamma$.
	Then there is a negative face $E$ of $\Gamma$ and 
	a positive face $F$ of $\Gamma$ such that
	$\widetilde{\tau}(t_1) \in E\times \{1\}$,
	$\widetilde{\tau}(t_2) \in F \times \{0\}$.
	By Definition \ref{I_F}, Definition \ref{psi} and Definition \ref{piece value}, 
	we have
	$i_0(E\times \{1\}) = h_E(1) = N_0 < N_1 = h_F(0) < i_0(F \times \{0\})$.
	It follows that 
	$i_0(\widetilde{\tau}(t_1)) < i_0(\widetilde{\tau}(t_2))$.
	
	Thus 
	$i_0(\widetilde{\tau}(1)) > i_0(\widetilde{\tau}(0))$,
	and therefore
	$\widetilde{\tau}(0), \widetilde{\tau}(1)$ are distinct,
	which implies that
	$\tau(I)$ is an essential simple closed curve in $M$.
\end{remark}

\subsection{Testing if $\mathcal{F}$ can extend to
a taut foliation of $M$}\label{subsection 4.6}

We already have a foliation $\mathcal{F}$ of $M_0 = N(B)$.
Now we consider: 
when can $\mathcal{F}$ extend to a taut foliation of $M$?

\begin{prop}\label{extendability}
	$\mathcal{F}$ can extend to 
	a taut foliation in $M$
	if and only if
	either of the following equivalent conditions holds:
	
	(1)
	$\partial_v N(B)$ is connected
	(then it is an annulus).
	
	(2)
	The union of 
	positive faces of $\Gamma$ is a single $2$-disk.
	Then the union of 
	negative faces of $\Gamma$ is also a single $2$-disk
	(see Figure \ref{extending}).
\end{prop}

\begin{figure}\label{extending}
	\includegraphics[width=0.6\textwidth]{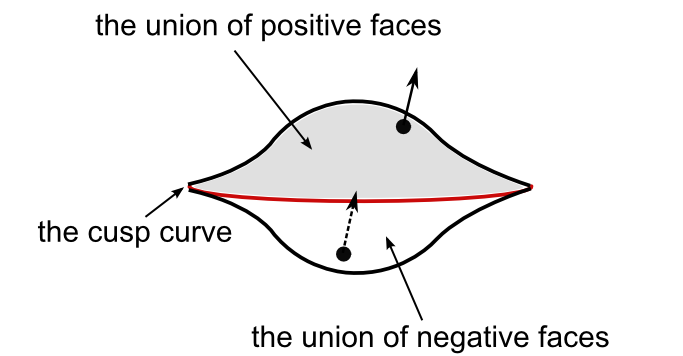}
	\caption{The union of positive faces/negative faces,
		and the cusp curve on $\Gamma$.}
\end{figure}

\begin{proof}
	We prove (1) at first.
	Let $\mathcal{G} = M - Int(N(B))$.
	Then $\mathcal{G}$ is a compact $3$-ball,
	and each component of 
	$\partial_v N(B)$ is an annulus that lies on $\partial \mathcal{G}$.
	It's clear that
	$\mathcal{F}$ can extend to a foliation of $M$ if and only if
	$\partial_v N(B)$ is a connected annulus 
	(then the extending foliation is obtained by
	filling $\mathcal{G}$ with horizontal disks).
	Now suppose that $\mathcal{F}$ can extend to a foliation in $M$.
	Similar to Lemma \ref{taut},
	we can choose a simple closed curve 
	transverse to the extending foliation that
	meets every leaf.
	Hence the extending foliation is taut.

	Now we show (1)$\iff$(2).	
	Let $\widetilde{\mathcal{G}}$ be the component of
	$p^{-1}(\mathcal{G})$ such that
	$\widetilde{\mathcal{G}} \subseteq \Gamma$.
	Then the map $\widetilde{\pi}: \widetilde{N(B)} \to \widetilde{B}$ collapses
    $\partial_v \widetilde{N(B)} \cap \partial \widetilde{\mathcal{G}}$ into
    the cusp curves on $\partial \Gamma$.
    Thus
    $\partial_v N(B)$ is connected if and only if
    there is exactly one cusp curve on $\partial \Gamma$.
    Recall from Definition \ref{orienting},
    every positive face (resp. negative face) of $\Gamma$ has 
    positive orientation
    (resp. negative orientation) with respect to $\Gamma$,
    as a branch sector of $\widetilde{B}$.
    Thus
    $$\bigcup_{l \text{ is a cusp curve on } \partial \Gamma} l = 
    \partial(\bigcup_{F \text{ is a positive face of } \Gamma} F),$$
    i.e.
    the union of cusp curves on $\partial \Gamma$ is the boundary of
    the union of positive faces of $\Gamma$.
    It follows that
    (1)$\iff$(2).
\end{proof}

\begin{lm}\label{induced transverse structure}
	Suppose that $\mathcal{F}$ can extend to a taut foliation
	$\mathcal{F}_1$ of $M$.
	Then $\mathcal{F}_1$ has a transverse $(\pi_1(M),\mathbb{R})$ structure.
\end{lm}
\begin{proof}
	Induced from the map $i_{des}: L \to \mathbb{R}$ 
	(cf. Lemma \ref{immersion}) directly.
\end{proof}

As in Definition \ref{order-domain pair}, 
$(<,\Gamma)$ is an order-domain pair of $M$.
We call $\mathcal{F}$ a \emph{resulting foliation} of $(<,\Gamma)$.
We complete the proof of Theorem \ref{extendable} in the remainder of this subsection:

\begin{extendable}
	(a)
	$\mathcal{F}$ is uniquely determined by $(<,\Gamma)$, 
	up to blowing-up/down.
	
	(b)
	$\mathcal{F}$ can extend to a taut foliation of $M$ if and only if
	the union of positive faces of $\Gamma$ is a single $2$-disk.
	Moreover,
	if $\mathcal{F}$ can extend to a taut foliation $\mathcal{F}_1$ of $M$,
	then $\mathcal{F}_1$ has a transverse $(\pi_1(M),\mathbb{R})$ structure.
	
	(c)
	$\mathcal{F}$ can extend to an $\mathbb{R}$-covered foliation of $M$ if and only if 
	the union of positive faces of $\Gamma$ is a single $2$-disk,
	and both of
	$\bigcup_{g \in G, g < 1}t_g(\Gamma)$,
	$\bigcup_{g \in G, g \geqslant 1}t_g(\Gamma)$ are connected.
\end{extendable}

To complete Theorem \ref{extendable},
it remains to discuss the following:
(1)
the case where $\mathcal{F}$ can extend to an $\mathbb{R}$-covered foliation of $M$
(cf. Remark \ref{R-covered}),
(2)
$\mathcal{F}$ is uniquely determined by $(<,\Gamma)$,
up to blowing-up/down (cf. Remark \ref{determined}).

\begin{remark}\rm\label{R-covered}
	Suppose that $\mathcal{F}$ can extend to a taut foliation $\mathcal{F}_1$ in $M$.
	$i_{des}$ takes some leaves of $\widetilde{\mathcal{F}}$ to
	the point $N_0 \in \mathbb{R}$ as chosen in Definition \ref{rho}.
	Notice that the $\mathcal{F}_1$ is $\mathbb{R}$-covered if and only if 
	$i^{-1}_{des}(N_0)$ is a single leaf of $\widetilde{\mathcal{F}}$.
	Through the proof of Lemma \ref{edge relation},
	we can observe that
	$\widetilde{\pi}: \widetilde{N(B)} \to \widetilde{B}$ takes
	$i^{-1}_{des}(N_0)$ to $\partial (\bigcup_{g \in G, g < 1} t_g(\Gamma))$.
	Thus,
	$i^{-1}_{des}(N_0)$ is a single leaf if and only if
	$\partial (\bigcup_{g \in G, g < 1} t_g(\Gamma))$ is connected,
	i.e. 
	both of
	$\bigcup_{g \in G, g < 1}t_g(\Gamma)$,
	$\bigcup_{g \in G, g \geqslant 1}t_g(\Gamma)$ are connected.
	Therefore,
	$\mathcal{F}$ can extend to an $\mathbb{R}$-covered foliation of $M$ if and only if 
	Proposition \ref{extendability} holds and
	both of
	$\bigcup_{g \in G, g < 1}t_g(\Gamma)$,
	$\bigcup_{g \in G, g \geqslant 1}t_g(\Gamma)$ are connected.
\end{remark}

\begin{remark}\rm\label{determined}
	In our construction,
	$\mathcal{F}$ only depends on the following information:
	
	(1)
	The branched surface $B$ as given in Subsection \ref{subsection 3.2}.
	
	(2)
	The action $\{\rho_g: \mathbb{R} \to \mathbb{R} \mid g \in G\}$ and
	the closed interval $[N_0,N_1] \subseteq \mathbb{R}$ as given in 
	Definition \ref{rho},
	which have following property:
	for any $g,h \in G$ with $g < h$ and $i,j \in \{0,1\}$,
	$\rho_g(N_i) < \rho_h(N_j)$.
	
	Notice that $B$ is homeomorphic to $p(\partial \Gamma)$ and
	has a co-orientation on its branch sectors which is determined by``$<$''.
	Next, we explain that $\{\rho_g: \mathbb{R} \to \mathbb{R} \mid g \in G\}$ 
	is determined by ``$<$'', up to blowing-up/down
	and conjugation by $Homeo_+(\mathbb{R})$.
	
	Let $\mathcal{U}$ be the set of components of $\mathbb{R} - \bigcup_{g \in G}\rho_g([N_0,N_1])$ 
	of the forms $[a,b)$, $(a,b]$, $[a,b]$.
	Notice that all elements of $\mathcal{U}$ have disjoint closure,
	and $\mathcal{U}$ is a countable set
	(since $\mathbb{R}$ contains at most countably many disjoint intervals).
	We blow-down $\bigcup_{J \in \mathcal{U}} \overline{J}$ in $\mathbb{R}$.
	Notice that $\mathcal{U}$ is invariant under the action
	$\{\rho_g \mid g \in G\}$,
	so this blowing-down induces a new action from $\{\rho_g \mid g \in G\}$.
	We denote the induced action by
	$\{\rho^{*}_{g} \mid g \in G\}$.
	We claim that $\{\rho^{*}_{g} \mid g \in G\}$ is uniquely determined by ``$<$'',
	up to conjugacy in $Homeo_+(\mathbb{R})$.
	To see this,
	we describe $\mathbb{R} - \bigcup_{g \in G}\rho^{*}_{g}([N_0,N_1])$ as follows.
	
	Let $A = \mathbb{R} - \bigcup_{g \in G}\rho^{*}_{g}([N_0,N_1])$.
	Let $\Omega$ be the set of $C \subseteq G$ ($C \ne \emptyset, G$) such that:
	if $h \in C$, then $g \in C$ for all $g \in G$ with $g < h$.
	Then there is a well-defined map $f: A \to \Omega$ such that
	$f(t) = \{g \in G \mid \rho^{*}_{g}(N_0) < t\}$ for all $t \in A$.
	For arbitrary $C \in \Omega$,
	we have the following observation:
	
	(1)
	Assume that both of $\max C, \min (G-C)$ exist.
	Then $f^{-1}(C)$ is an open interval.
	
	(2)
	Assume that exactly one of $\max C, \min (G-C)$ exists.
	Then $f^{-1}(C)$ does not exist
	(otherwise, $f^{-1}(C)$ is an interval of the form either $[a,b)$ or $(a,b]$).
	
	(3)
	Assume that both of $\max C, \min (G-C)$ do not exist.
	Then $f^{-1}(C)$ is a single point
	(otherwise, $f^{-1}(C)$ is an interval of the form $[a,b]$).
	
	By the above discussions,
	$A$ is determined by $\Omega$ (and therefore is determined by ``$<$''),
	up to $Homeo_+(\mathbb{R})$.
	Thus,
	$\{\rho^{*}_{g} \mid g \in G\}$ is uniquely determined by ``$<$'',
	up to conjugacy in $Homeo_+(\mathbb{R})$.
	$\{\rho^{*}_{g} \mid g \in G\}$ produces a new resulting foliation (denoted $\mathcal{F}^{*}$) of $(<,\Gamma)$
	through the process in Subsection \ref{subsection 4.4},
	which is uniquely determined by $(<,\Gamma)$.
	
	At last,
	we consider the change between $\mathcal{F}^{*}$ and 
	the original resulting foliation $\mathcal{F}$.
	Recall that we blow-down the set
	$\bigcup_{J \in \mathcal{U}} \overline{J}$ in $\mathbb{R}$ to make
	$\{\rho_g \mid g \in G\}$ be $\{\rho^{*}_{g} \mid g \in G\}$ and
	make $\mathcal{F}$ be $\mathcal{F}^{*}$.
	For each $J \in \mathcal{U}$,
	we have
	$\rho_g(N_i) \notin J$ for all $g \in G$ and $i \in \{0,1\}$,
	and thus $i^{-1}_{des}(J) \cap \partial_h \widetilde{N(B)} = \emptyset$.
	So $i^{-1}_{des}(\overline{J})$
	is a union of pockets of leaves in $\widetilde{\mathcal{F}}$
	(which is necessarily a countable union).
	So the above blowings-down in $\mathbb{R}$ 
	make $\widetilde{\mathcal{F}}$ be blown-down in $\widetilde{M}$
	in a $\pi_1$-equivariant way,
	and therefore make $\mathcal{F}$ be blown-down in $M$.
	As a result,
	we can blow-down $\mathcal{F}$ countably many times to
	obtain $\mathcal{F}^{*}$.
	
	Thus,
	all resulting foliations of $(<,\Gamma)$ are same up to blowing-up/down.
\end{remark}

Combining Proposition \ref{extendability},
Lemma \ref{induced transverse structure},
Remark \ref{R-covered},
Remark \ref{determined},
we can complete
the proof of Proposition \ref{extendable}.

\begin{figure}\label{surface case}
	\centering
	\subfigure[]{
	\includegraphics[width=0.33\textwidth]{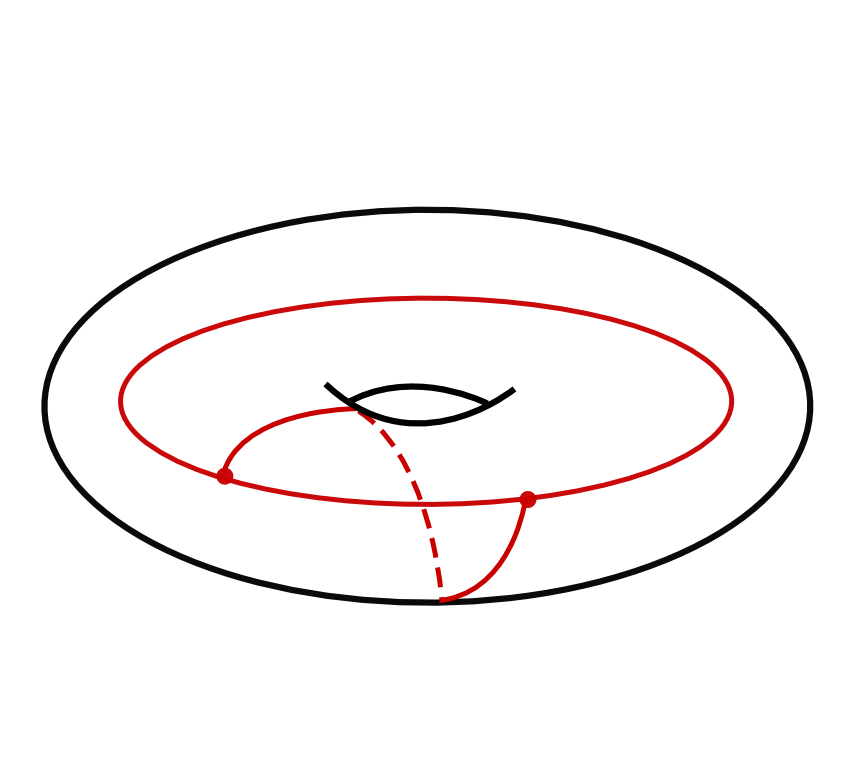}}
	\subfigure[]{
	\includegraphics[width=0.33\textwidth]{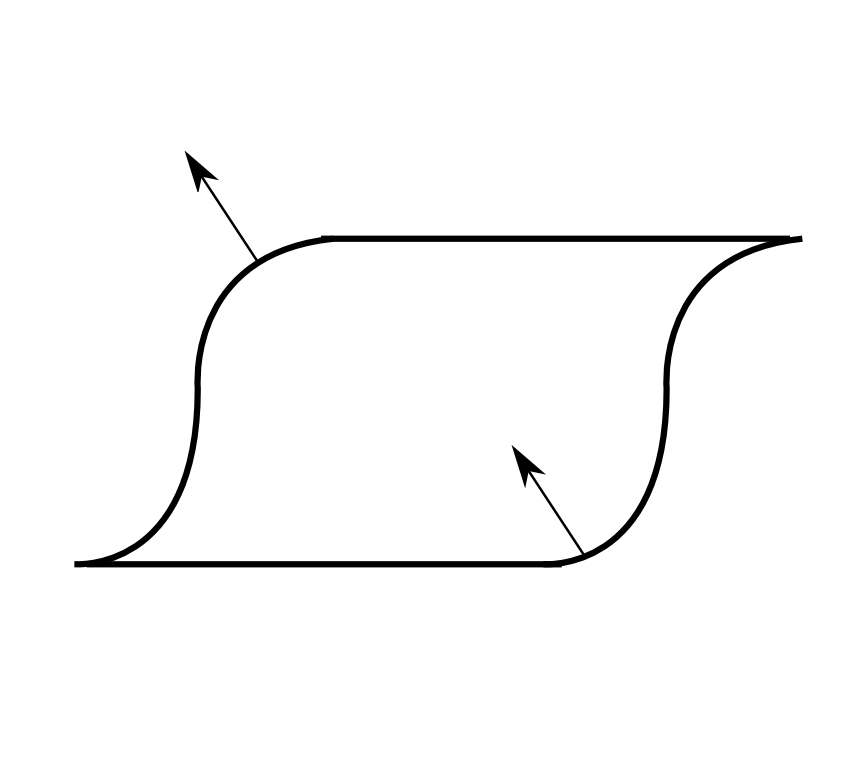}}
	\centering
	\subfigure[]{
	\includegraphics[width=0.33\textwidth]{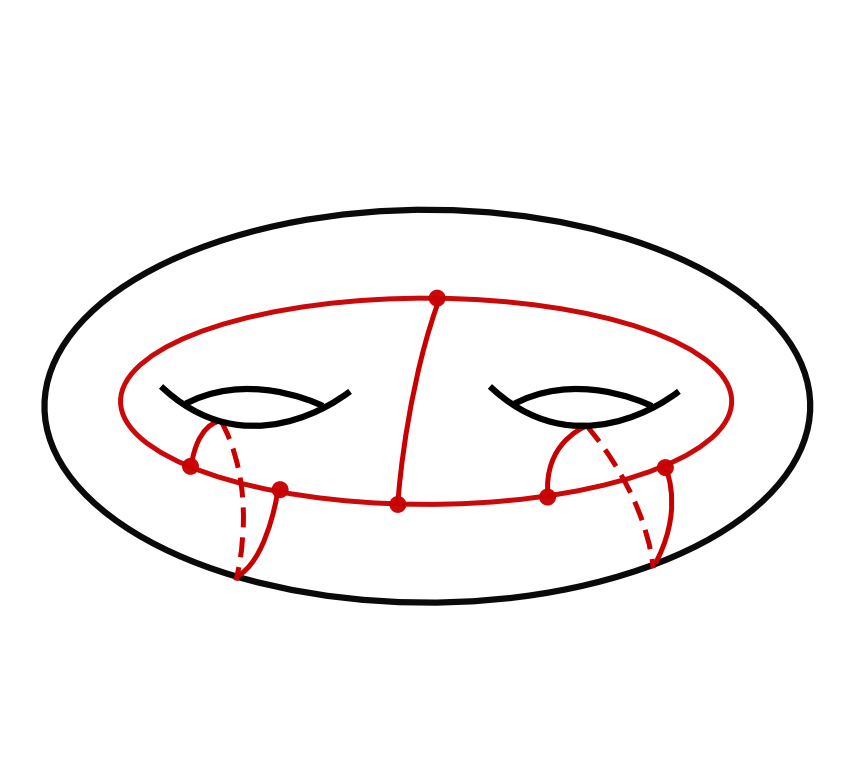}}
	\subfigure[]{
	\includegraphics[width=0.33\textwidth]{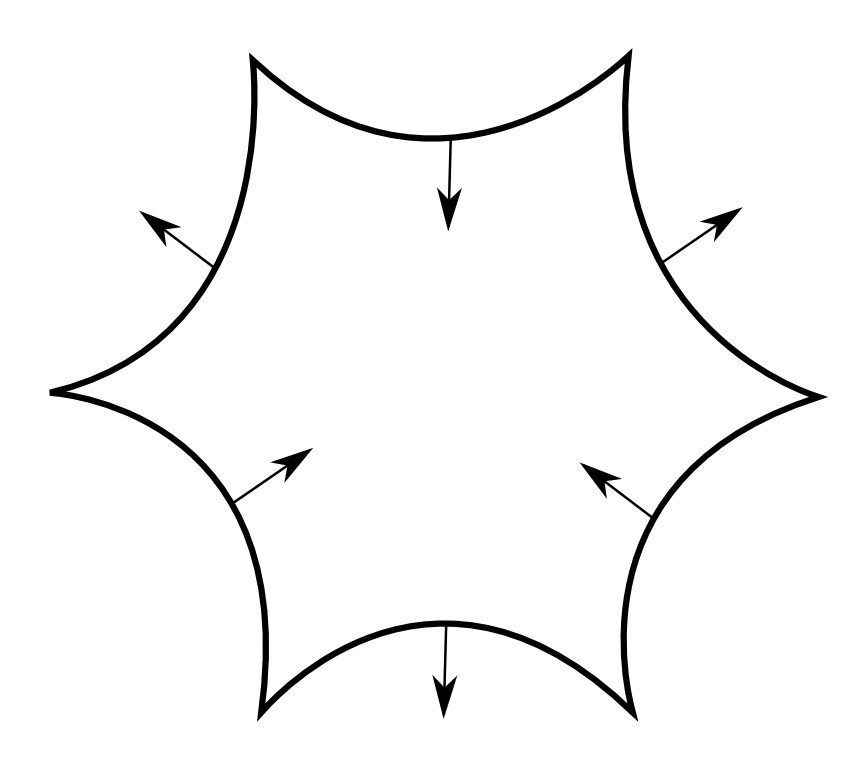}}
	\caption{(a), (b): Suppose that $S$ is a torus.
		Then the trivalent graph $p(\partial \Gamma)$ has 
		two $3$-valent vertices exactly.
		We deform $p(\partial \Gamma)$ to 
		the train track $\tau$,
		then the boundary of
		$S \setminus \setminus \tau$ has two cusps.
		(c), (d):
		Suppose that $S$ is a closed orientable surface $S$ of genus $2$.
		Then the trivalent graph $p(\partial \Gamma)$ has 
		six $3$-valent vertices.
		Thus $\tau$ has six cusps,
		and the boundary of
		$S \setminus \setminus \tau$ has six cusps.}
\end{figure}

\subsection{The $2$-dimensional case}\label{subsection 4.7}

Let $S$ be a closed orientable surface of genus $\geqslant 1$.
Then $\pi_1(S)$ is left orderable.
Let $p: \widetilde{S} \to S$
be the universal cover of $S$.
Fix an arbitrary left-invariant order ``$<$'' of $\pi_1(S)$.
We choose a fundamental domain $\Gamma$
of $S$ such that $p(\partial \Gamma)$
is a trivalent graph.
Similar to Subsection \ref{subsection 3.2},
we can orient the trivalent graph $p(\partial \Gamma)$ to a train track $\tau$
through the order ``$<$''.
Let $D$ be a compact $2$-disk on $S$.
Through the process in
Subsection \ref{subsection 4.1}$\sim$\ref{subsection 4.4},
we can construct an one-dimensional foliation $\mathcal{F}$ of 
$S - Int(D)$ that 
has a transverse $(\pi_1(S),\mathbb{R})$ structure.

Next, we consider the following two cases:

(a)
Suppose that $S$ is a torus.
As shown in Figure \ref{surface case} (a),
$p(\partial \Gamma)$ has two $3$-valent vertices.
We orient $p(\partial \Gamma)$ to a train track $\tau$,
then each $3$-valent vertex of $p(\partial \Gamma)$ produces
exactly one cusp of $\tau$.
So $\tau$ has two cusps,
and
$S \setminus \setminus \tau$ is a disk with two cusps on the boundary (cf. Figure \ref{surface case} (b)).
Similar to Proposition \ref{extendability},
we can extend $\mathcal{F}$ to a foliation of $S$ by 
filling $D$ with
horizontal segments.

(b)
Suppose that $S$ is a closed orientable surface
of genus $g \geqslant 2$.
There are $4g-2$ vertices of $p(\partial \Gamma)$ that have valency $3$.
Thus $\tau$ has $4g-2 \geqslant 6$ cusps,
and
$S \setminus \setminus \tau$ is a disk with at least $6$ cusps on the boundary.
So $\mathcal{F}$ can not extend to a foliation of $S$.
On the other hand,
we can ensure that
$S$ admits no foliation since $\chi(S) \ne 0$.
We can contract $D$ to a single point and
make $\mathcal{F}$ be deformed to 
a co-orientable singlular foliation on $S$ with 
a single $(4g-2)$-pronged singular point.
	
\section{$\mathbb{R}$-covered foliations constructed in our process}\label{section 5}

We prove Theorem \ref{process} in this section.
To begin with,
we review some backgrounds.
A taut foliation has \emph{hyperbolic leaves} if
there is a continuously varying leafwise metric on its leaves so that
every leaf is locally isometric to $\mathbb{H}^{2}$
(cf. \hyperref[Can]{[Can]}).
The following theorem is shown by Candel in \hyperref[Can]{[Can]}:

\begin{thm}[Candel]
	Every taut foliation of an irreducible atoroidal $3$-manifold has hyperbolic leaves.
\end{thm}

Boyer and Hu prove that taut foliations in rational homology $3$-spheres also have hyperbolic leaves
(cf. \hyperref[BH]{[BH]}):

\begin{thm}[Boyer-Hu]
	Every taut foliation of a rational homology $3$-sphere has hyperbolic leaves.
\end{thm}

\begin{defn}\rm
	Let $\mathcal{F}$ be an $\mathbb{R}$-covered foliation of a closed $3$-manifold $N$.
	Let $\widetilde{\mathcal{F}}$ be
	the pull-back foliation of $\mathcal{F}$ in the universal cover of $N$.
	Let $X$ be a vector field transverse to $\mathcal{F}$ and
	$\widetilde{X}$ the pull-back vector field of $X$ in the universal cover of $N$.
	$X$ is called 
	\emph{regulating} if every orbit of 
	$\widetilde{X}$ intersects every leaf of $\widetilde{\mathcal{F}}$ exactly once. 
\end{defn}

In \hyperref[Cal2]{[Cal2]},
Calegari proves the following theorem:

\begin{thm}[Calegari]\label{hyperbolic leaf}
	Every $\mathbb{R}$-covered foliation with hyperbolic leaves admits
	a regulating transverse vector field.
\end{thm}

It follows that

\begin{cor}\label{regulating}
	Let $N$ be either an irreducible atoroidal $3$-manifold or
	a rational homology $3$-sphere.
	Then every $\mathbb{R}$-covered foliation of $N$ admits
	a regulating transverse vector field.
\end{cor}

Now we begin to prove Theorem \ref{process}.
Let $M$ be the closed irreducible $3$-manifold
as given in Notation \ref{M}.
We assume further that
$M$ is either an atoroidal $3$-manifold or a rational homology $3$-sphere,
and
$M$ admits an $\mathbb{R}$-covered foliation $\mathcal{F}_0$.
For the reader's convenience,
we restate Theorem \ref{process} as follows:

\begin{process}
	There is a resulting foliation $\mathcal{F}$ of $M - Int(B^{3})$ obtained from
	the process of Theorem \ref{theorem 1} such that:
	
	(1)
	$\mathcal{F}$ can extend to 
	an $\mathbb{R}$-covered foliation $\mathcal{F}_{extend}$ of $M$.
	
	(2)
	$\mathcal{F}_0$ can be recovered from doing a collapsing operation on $\mathcal{F}_{extend}$.
\end{process}

This section is organized as follows:

\begin{sketch}\rm    
	(a)
	In Subsection \ref{subsection 5.1},
	we show that
	there exists an $\mathbb{R}$-covered foliation $\mathcal{F}_1$ of $M$ such that
	(1)
	$\mathcal{F}_0$ can be obtained from doing 
	a collapsing operation for $\mathcal{F}_1$,
	(2)
	$\mathcal{F}_1$ has a $2$-plane leaf $\lambda$.
	
	(b)
	In Subsection \ref{subsection 5.2},
	we blow-up the $2$-plane leaf $\lambda$ of $\mathcal{F}_1$ to
	obtain an $\mathbb{R}$-covered foliation $\mathcal{F}_2$.
	We choose an order-domain pair $(<,\Gamma)$ of $M$ and
	show that: 
	there is a resulting foliation $\mathcal{F}$ of $(<,\Gamma)$
	in $M - Int(B^{3})$ that
	can extend to $\mathcal{F}_2$.
    The combination 
    $\mathcal{F} \rightsquigarrow \mathcal{F}_2 \rightsquigarrow
    \mathcal{F}_1 \rightsquigarrow \mathcal{F}_0$ is
    the process as required in Theorem \ref{process}.
\end{sketch}

Throughout this section,
we adopt the following notations:

\begin{notation}\rm\label{notation}
Let $\widetilde{l}$ be a leaf of a $\pi_1$-equivariant foliation of $\widetilde{M}$.
We will always denote by
$Stab_G(\widetilde{l}) = \{g \in G \mid t_g(\widetilde{l}) = \widetilde{l}\}$
the stabilizer subgroup of $G$
with respect to $\widetilde{l}$.
\end{notation}

\subsection{Perturbing a pocket through an action}\label{subsection 5.1}

To begin with,
we review the existence of fundamental domains for
arbitrary surface in its universal cover:

\begin{lm}
	Let $S$ be an orientable surface
	($S$ may be a surface of infinite-type).
	Let $\widetilde{S}$ be the universal cover of $S$.
	There is a (possibly non-compact) fundamental domain
	$\Omega$ of $S$ in $\widetilde{S}$ such that
	every point in the boundary of
	$\Omega$ is contained in finitely many members of
	$$\{g(\Omega) \mid g \text{ is a deck transformation of } \widetilde{S}\}.$$
\end{lm}
\begin{proof}
	This follows from \hyperref[Ri]{[Ri]} and \hyperref[A]{[A]}.
\end{proof}

Let $\widetilde{\mathcal{F}_0}$ be the pull-back foliation of $\mathcal{F}_0$ in $\widetilde{M}$.
Let $l$ be a leaf of $\mathcal{F}_0$.
Let
$\widetilde{l}$ be a leaf of $\widetilde{\mathcal{F}_0}$ such that $p(\widetilde{l}) = l$.
By Novikov's Theorem (\hyperref[N]{[N]}, \hyperref[Ro]{[Ro]}),
$\widetilde{l}$ is a universal cover of $l$
(where the deck transformations on $\widetilde{l}$ are induced from
the deck transformations of $\widetilde{M}$ which
preserve $\widetilde{l}$).
So there exists a fundamental domain of 
$l$ in $\widetilde{l}$ as above.

\begin{defn}\rm\label{fundamental domain}
	Let
	$K = Stab_G(\widetilde{l}) \subseteq G$.
	Then
	$K$ acts on $\widetilde{l}$ by
	the restriction of 
	$\{t_g \mid g \in K\}$ to $\widetilde{l}$.
	We choose a (possibly non-compact) fundamental domain $\Sigma$ for
	this action on $\widetilde{l}$ such that
	every point in the boundary of
	$\Sigma$ is contained in finitely many members of
	$\{t_g(\Sigma) \mid g \in K\}$.
\end{defn}

\begin{defn}\rm\label{blowing-up}
	We blow-up the leaf $l$ of $\mathcal{F}_0$ to obtain
	a foliation $\mathcal{F}^{blow}_{0}$ of $M$.
	Let $\widetilde{\mathcal{F}^{blow}_{0}}$ be its pull-back foliation in $\widetilde{M}$.
	Then the leaf $l$ of $\mathcal{F}_0$ is replaced by a blown-up pocket
	$l \times I$ in $\mathcal{F}^{blow}_{0}$,
	and the leaf $\widetilde{l}$ of $\widetilde{\mathcal{F}_0}$ is replaced by
	a blown-up pocket $\widetilde{l} \times I$ in $\widetilde{\mathcal{F}^{blow}_{0}}$.
	We can regard $\widetilde{l} \times I$ as the universal cover of $l \times I$.
	As in Definition \ref{fundamental domain},
	$\widetilde{l} \times I$ can be decomposed to
	$\{t_g(\Sigma) \times I \mid g \in K\}$.
	Here,
	$t_g(\Sigma)$ denotes a region in the fibered surface $\widetilde{l}$ of $\widetilde{l} \times I$.
\end{defn}

We fix an effective action $\{s_g: I \to I \mid g \in G\}$ of $G$ on $I$
such that 
$s_g(\frac{1}{2}) \ne \frac{1}{2}$ for every $g \in G - \{1\}$.
Now we glue the copies of
$\{t_g(\Sigma) \times I \mid g \in K\}$ to make the sheets
$\{t_g(\Sigma) \times \{t\} \mid g \in K, t \in I\}$
be glued to a foliation in
$\widetilde{l} \times I$:

\begin{construction}\rm\label{F_1}
	We decompose 
	$\widetilde{l} \times I$ to
	$\{t_g(\Sigma) \times I \mid g \in K\}$.
	Now we consider arbitrary distinct elements $r,h \in K$ with
	$t_r(\Sigma) \cap 
	t_h(\Sigma) \ne \emptyset$.
	Let $e$ denote $t_r(\Sigma) \cap t_h(\Sigma)$,
	and let
	$e \times I$ denote $(t_r(\Sigma) \times I) \cap (t_h(\Sigma) \times I)$.
	Let
	$\delta_{(r,h)}: e \times I \to e \times I$
	be the map defined by 
	$$\delta_{(r,h)}(u,v) = (u,s_{h^{-1}r}(v)),
	\forall u \in e, v \in I.$$
	We glue the copy of
	$e \times I$ in
	$t_r(\Sigma) \times I$
	to the copy of
	$e \times I$ in
	$t_h(\Sigma) \times I$
	through the map
	$\delta_{(r,h)}$
	(cf. Figure \ref{holonomy}).
	Now for each $t \in I$,
	the segment
	$e \times \{t\} \subseteq 
	t_r(\Sigma) \times \{t\}$
	is glued with the segment
	$e \times 
	\{s_{h^{-1}r}(t)\} \subseteq 
	t_h(\Sigma) \times \{s_{h^{-1}r}(t)\}$.
\end{construction}

We claim that:
for arbitrary three distinct members in
$\{t_g(\Sigma) \times I \mid g \in K\}$ that have nonempty common-intersection,
all the gluing maps we used above for them are 
compatible in their common-intersection.
Now suppose that $v$ is a point in the fibered surface $\widetilde{l}$ of $\widetilde{l} \times I$ such that
$v \in$ $t_r(\Sigma)$, $t_h(\Sigma)$, $t_j(\Sigma)$ for distinct $r,h,j \in K$.
By Construction \ref{F_1},
for each $t \in I$,
the gluing maps between
$t_r(\Sigma) \times I$,
$t_h(\Sigma) \times I$,
$t_j(\Sigma) \times I$ glue
the three points 
$(v,t) \in t_r(\Sigma) \times I$,
$(v,s_{h^{-1}r}(t)) \in t_r(\Sigma) \times I$,
$(v,s_{j^{-1}r}(t)) \in t_r(\Sigma) \times I$ together
(cf. Figure \ref{holonomy}).
So
these gluing maps are compatible in
$v \times I$.

\begin{figure}\label{holonomy}
	\includegraphics[width=0.6\textwidth]{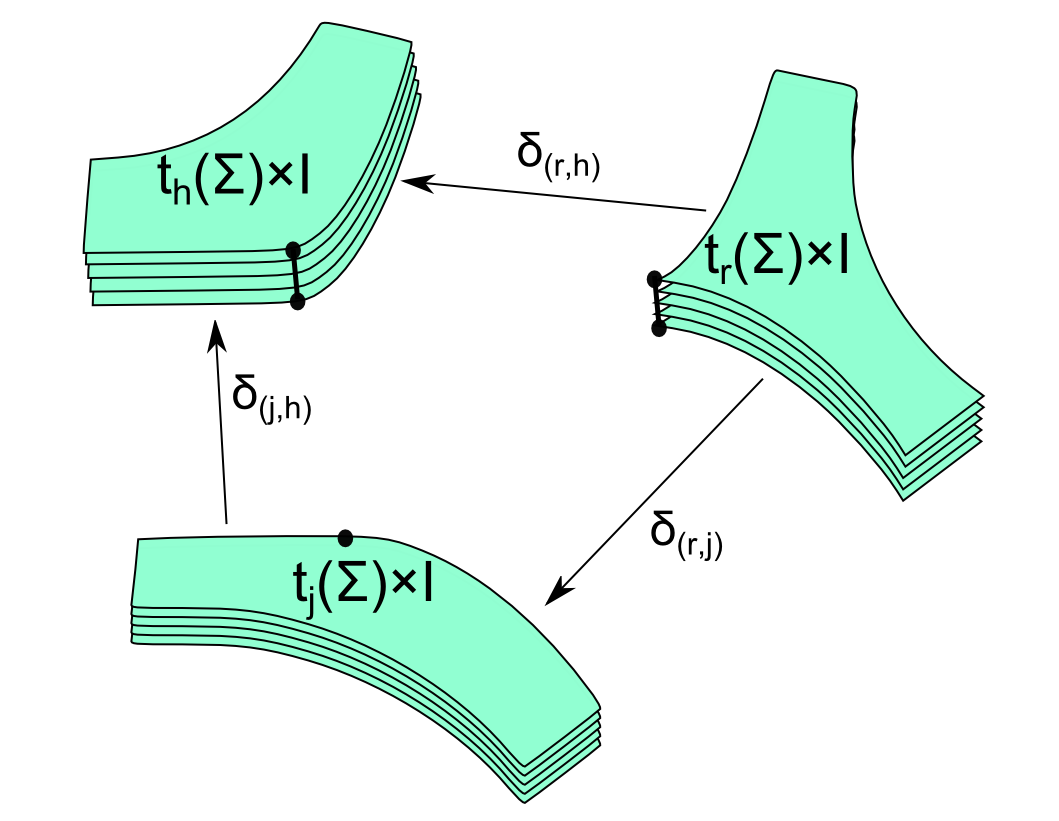}
	\caption{Given $r,h,j \in K$ such that the three regions
		$t_r(\Sigma)$,
		$t_h(\Sigma)$,
		$t_j(\Sigma)$ in the fibered surface $\widetilde{l}$ of $\widetilde{l} \times I$ have a common point,
		then for every $t \in I$,
		the three horizontal sheets
		$t_r(\Sigma) \times \{t\}$,
		$t_h(\Sigma) \times \{s_{h^{-1}r}(t)\}$,
		$t_j(\Sigma) \times \{s_{j^{-1}r}(t)\}$
		are attached.}
\end{figure}

We denote by $t^{*}_{g}: \widetilde{l} \times I \to \widetilde{l} \times I$ ($g \in K$) 
the deck transformation of $\widetilde{l} \times I$ which
takes $\Sigma \times I$ to $t_g(\Sigma) \times I$.
And we call $t_g(\Sigma) \times \{t\}$
a \emph{horizontal sheet} for all $g \in K$, $t \in I$.
In the following,
we show that all horizontal sheets are glued to
a foliation of $\widetilde{l} \times I$ which is equivariant under 
$\{t^{*}_{g}: \widetilde{l} \times I \to \widetilde{l} \times I \mid g \in K\}$:

\begin{lm}\label{leaf}
	$\{t_g(\Sigma) \times \{t\} \mid g \in K, t \in I\}$ is glued to a foliation
	(denoted $\widetilde{\mathcal{F}^{pert}}$) of $\widetilde{l} \times I$.
	For every $t \in I$,
	$\bigcup_{g \in K}(t_g(\Sigma) \times \{s^{-1}_{g}(t)\})$ is a leaf of $\widetilde{\mathcal{F}^{pert}}$.
	And 
	$t^{*}_{g}(\widetilde{\mathcal{F}^{pert}}) = \widetilde{\mathcal{F}^{pert}}$ for every $g \in K$.
\end{lm}

\begin{proof}
	Let $t \in I$.
	Let $h,r$ be distinct elements of $K$ such that
	$t_h(\Sigma) \cap t_r(\Sigma) \ne \emptyset$.
	Let $e = t_h(\Sigma) \cap t_r(\Sigma)$.
	Since 
	$s_{h^{-1}r}(s^{-1}_{r}(t)) = 
	s^{-1}_{h}(t)$,
	the two horizontal sheets
	$t_r(\Sigma) \times \{s^{-1}_{r}(t)\}$,
	$t_h(\Sigma) \times \{s^{-1}_{h}(t)\}$ are 
	glued along the two segments
	$$e \times \{s^{-1}_{r}(t)\} \subseteq
	t_r(\Sigma) \times \{s^{-1}_{r}(t)\},$$
	$$e \times \{s^{-1}_{h}(t)\} \subseteq
	t_h(\Sigma) \times \{s^{-1}_{h}(t)\}$$ as in
	Construction \ref{F_1}.
	It's clear that all horizontal sheets in
	$\{t_g(\Sigma) \times \{s^{-1}_{g}(t)\} \mid g \in K\}$ are glued to a leaf
	(we denote it by $\widetilde{l} \times \{t\}$).
	$\{\widetilde{l} \times \{t\} \mid t \in I\}$ forms the foliation $\widetilde{\mathcal{F}^{pert}}$ of $\widetilde{l} \times I$.
	
	For all $h \in K$, $t \in I$,
	we have
	\begin{align}
	t^{*}_{h}(\widetilde{l} \times \{t\}) &
	= t^{*}_{h}(\bigcup_{g \in K}
	(t_g(\Sigma) \times \{s^{-1}_{g}(t)\})) \nonumber \\ &
	= \bigcup_{g \in K}
	(t_{hg}(\Sigma) \times \{s^{-1}_{g}(t)\}) \nonumber \\ &
	= \bigcup_{g \in K}
	(t_{g}(\Sigma) \times \{s_{g^{-1}h}(t)\}) \nonumber \\ &
	= \bigcup_{g \in K}
	(t_{g}(\Sigma) \times \{s^{-1}_{g}(s_h(t))\}) \nonumber \\ &
	= \widetilde{l} \times \{s_h(t)\}. \label{equation}
	\end{align}
	Thus $t^{*}_{h}(\widetilde{\mathcal{F}^{pert}}) = \widetilde{\mathcal{F}^{pert}}$.
\end{proof}

$\widetilde{\mathcal{F}^{pert}}$ descends to a foliation
$\mathcal{F}^{pert}$ of $l \times I$.
We can replace $\mathcal{F}^{blow}_{0} \mid_{l \times I}$ by $\mathcal{F}^{pert}$ to
obtain a foliation of $M$.
We denote this foliation by $\mathcal{F}_1$, 
and we
denote its pull-back foliation in $\widetilde{M}$ by $\widetilde{\mathcal{F}_1}$.
For every $g \in G$,
the pocket $t_g(\widetilde{l} \times I)$ in $\widetilde{\mathcal{F}^{blow}_{0}}$ is replaced by
$t_g(\widetilde{\mathcal{F}^{pert}})$ in $\widetilde{\mathcal{F}_1}$.
So $\mathcal{F}_1$ is still an $\mathbb{R}$-covered foliation.

\begin{lm}\label{2-plane}
	There is a $2$-plane leaf of $\mathcal{F}_1$.
\end{lm}
\begin{proof}
	Let $h \in K - \{1\}$.
	By Equation \ref{equation} (cf. the proof of Lemma \ref{leaf}),
	we have $t^{*}_{h}(\widetilde{l} \times \{t\}) = \widetilde{l} \times \{s_h(t)\}$ 
	for every $t \in I$.
	So
	$t^{*}_{h}$ takes the leaf $\widetilde{l} \times \{\frac{1}{2}\}$ to
	the leaf $\widetilde{l} \times \{s_h(\frac{1}{2})\}$.
	Since $s_h(\frac{1}{2}) \ne \frac{1}{2}$,
	we have
	$t^{*}_{h}(\widetilde{l} \times \{\frac{1}{2}\}) \ne \widetilde{l} \times \{\frac{1}{2}\}$.
	By Novikov's Theorem (\hyperref[N]{[N]}, \hyperref[Ro]{[Ro]}),
	$\widetilde{l} \times \{\frac{1}{2}\}$ descends to a $2$-plane leaf of $\mathcal{F}_1$.
\end{proof}

\begin{remark}\rm\label{still R-covered}
	Recall from Definition \ref{collapsing operation} that a collapsing operation for a taut foliation is to
	replace a product region $S \times I$ 
	(where $S \times \{0\}$, $S \times \{1\}$ are leaves) 
	by a single leaf $S$.
	$\mathcal{F}_0$ can be obtained from doing a collapsing operation for $\mathcal{F}_1$,
	which replaces $\mathcal{F}_1 \mid_{l \times I}$ by a single leaf $l$.
\end{remark}

\subsection{The proof of Theorem \ref{process}}\label{subsection 5.2}

We complete the proof of Theorem \ref{process} in this subsection.
As shown in Lemma \ref{2-plane},
there exists a $2$-plane leaf of $\mathcal{F}_1$ contained in the region $l \times I$.
We denote this $2$-plane leaf by $\lambda$.

\begin{fact}\rm\label{finite}
	Let $C$ be an embedded compact disk in $\lambda$.
	Let $\tau$ be a transversal (with endpoints) of $\mathcal{F}_1$.
	Then $|\tau \cap C|$ is finite.
\end{fact}
\begin{proof}
	Assume that $|\tau \cap C|$ is infinite.
	Then there is a limit point $q$ of $\tau \cap C$ in $\tau$.
	Since $\tau, C$ are closed subsets of $M$,
	$\tau \cap C$ is closed.
	So $q \in \tau \cap C$.
	There exists an open neighborhood $U_q$ of $q$ in $C$ and
	an open neighborhood $V_q$ of $q$ in $\tau$ such that
	$U_q \cap V_q = \{q\}$.
	Notice that $C - U_q$ is closed.
	So there exists an open neighborhood $N_q$ of $q$ in $M$ such that
	$N_q \cap (C - U_q) = \emptyset$.
	It's clear that $q$ is the only intersection point of $C$ and the segment $V_q \cap N_q \subseteq \tau$.
	So $q$ is an isolated point of $\tau \cap C$ in $\tau$,
	which is a contradiction.
	Therefore,
	$|\tau \cap C|$ is finite.
\end{proof}

We blow-up the leaf $\lambda$ of $\mathcal{F}_1$
to obtain a foliation $\mathcal{F}_2$ of $M$.
Then $\mathcal{F}_2$ is still an $\mathbb{R}$-covered foliation,
and we still can do a collapsing operation on $\mathcal{F}_2$ to obtain $\mathcal{F}_0$.
Now the leaf $\lambda$ of $\mathcal{F}_1$
is replaced by a pocket $\lambda \times I$ in $\mathcal{F}_2$.
We assume that a transversal of $\mathcal{F}_2$ from
$\lambda \times \{0\}$ to $\lambda \times \{1\}$ is positively oriented.
Let $\mathcal{L}$ be the essential lamination $\mathcal{F}_2 - \lambda \times (0,1)$.

\begin{prop}\label{3-ball}
	There is a branched surface $B$ of $M$ such that
	(1)
	$B$ fully carries $\mathcal{L}$,
	(2)
	$M \setminus \setminus B$ is a compact $3$-ball.
\end{prop}
\begin{proof}
	By Corollary \ref{regulating},
	there is a regulating vector field $X$ in $M$ transverse to
	$\mathcal{F}_2$.
	Let $\widetilde{X}$ be the pull-back vector field of $X$ in 
	$\widetilde{M}$.
	Let $\mathcal{O} = \{\text{orbits of } X\}$ and
	$\widetilde{\mathcal{O}} = \{\text{orbits of } \widetilde{X}\}$.
	We can assume
	$\mathcal{O} \mid_{\lambda \times I} =
	\{\{q\} \times I \mid q \in \lambda\}$.
	
	Let $\{U_\alpha \mid \alpha \in \Psi\}$
	(where $\Psi$ is an index set) be a finite union of product charts of
	$\mathcal{F}_1$ such that 
	$\bigcup_{\alpha \in \Psi} U_\alpha = M$.
	For each $\alpha \in \Psi$,
	we choose a component $\widetilde{U_\alpha}$ of
	$p^{-1}(U_\alpha)$,
	and we denote
	$\{\gamma \in \widetilde{\mathcal{O}} \mid 
	\gamma \cap \widetilde{U_\alpha} \ne \emptyset\}$ by $\widetilde{\mathcal{O}_\alpha}$.
	For any $\alpha \in \Psi$ and
	$\gamma \in \widetilde{\mathcal{O}_\alpha}$,
	we denote by $\gamma^{+}_{\alpha}$ (resp. $\gamma^{-}_{\alpha}$) 
	the component of
	$\gamma - \widetilde{U_\alpha}$ which contains 
	the positive infinity (resp. negative infinity) of $\gamma$.
	Notice that $X$ is regulating.
	For each $\alpha \in \Psi$,
	there is an open connected set $\widetilde{X_\alpha}$ in some component of $p^{-1}(\lambda \times \{0\})$
	such that:
	for every $\gamma \in \widetilde{\mathcal{O}_\alpha}$,
	$\gamma^{+}_{\alpha}$ intersects $\widetilde{X_\alpha}$.
	And there is an open connected set $\widetilde{Y_\alpha}$ (for each $\alpha \in \Psi$)
	in some component of $p^{-1}(\lambda \times \{1\})$
	such that:
	for every $\gamma \in \widetilde{\mathcal{O}_\alpha}$,
	$\gamma^{-}_{\alpha}$ intersects $\widetilde{Y_\alpha}$.
	
	Since $\lambda \times \{0\}, \lambda \times \{1\}$ are $2$-plane leaves of $\mathcal{F}_2$,
	there exists an embedded compact disk $C \subseteq \lambda$ such that:
	for each $\alpha \in \Psi$,
	$p(\widetilde{X_\alpha}) \subseteq Int(C) \times \{0\}$,
	$p(\widetilde{Y_\alpha}) \subseteq Int(C) \times \{1\}$
	(where $Int(C) \times \{0\}, Int(C) \times \{1\}$ 
	denote regions in $\lambda \times \{0\}$, $\lambda \times \{1\}$
	respectively).
	Then for arbitrary point $q \in M - Int(C \times I)$,
	the ray that starts at $q$ and goes along 
	the element of $\mathcal{O}$ that contains $q$
    (through either direction) must meet
	$Int(C) \times \{0\} \cup Int(C) \times \{1\}$ somewhere.
	
	Let $\mathcal{N} = M - Int(C \times I)$.
	Let $\mathcal{O}_1$ be the restriction of 
	$\mathcal{O}$ to $\mathcal{N}$.
	Then every element of $\mathcal{O}_1$ is
	a closed interval.
	Let $B = \mathcal{N} / \stackrel{\mathcal{O}_1}{\sim}$,
	which is a ``branched surface'' whose
	branch locus may not only have double-intersection singularities.
	For simplicity,
	we still call $B$ a branched surface,
	and we still adopt Notation \ref{branched surface notation} for $B$.
	Then $\mathcal{N}$ is a fibered neighborhood of $B$ with 
	$\{\text{interval fibers of } \mathcal{N}\} = \mathcal{O}_1$.
	In the following,
	we first explain that
	every point of $B$ is contained in finitely many branch sectors,
	and then explain that
	$B$ has finitely many branch sectors.
	
	Let $N(B) = \mathcal{N}$ with
	$\{\text{interval fibers of } N(B)\} = \mathcal{O}_1$,
	and let $\pi: N(B) \to B$ denote the collapsing map for $N(B)$.
	Then $\partial_h N(B) = (C \times \{0\}) \cup (C \times \{1\})$.
	Let $v \in B$.
	Let $J$ be the interval fiber of $N(B)$ such that $v = \pi(J)$.
	If $v$ is contained in infinitely many branch sectors of $B$,
	then $J$ must intersect $\partial_h N(B)$ infinitely many times
	(where the intersections of $Int(J)$ and $\partial_h N(B)$ are contained in 
	the boundary of $\partial_h N(B)$).
	By Fact \ref{finite},
	$J$ intersects $(C \times \{0\}) \cup (C \times \{1\})$ finitely many times.
	Thus,
	$v$ is contained in finitely many branch sectors of $B$.
	Now assume that $B$ has infinitely many branch sectors.
	We choose an interior point in each branch sector of $B$,
	and we denote their union by $P$.
	Then $\pi^{-1}(P) \cap (C \times \{1\})$ is an infinite set,
	and so there must be a limit point $q$ of the set $\pi^{-1}(P) \cap (C \times \{1\})$ in $C \times \{1\}$.
	However,
	$\pi(q)$ is contained in finitely many branch sectors of $B$.
	This is a contradiction.
	Thus $B$ has finitely many branch sectors.
	At last,
	we can isotope the branch locus of $B$ so that
	it only has double-intersection singularities.
	Then $B$ is a branched surface as in Definition \ref{branched surface definition}.
	
	Since $\mathcal{L} \subseteq N(B)$ and
	$\mathcal{L}$ intersects $\mathcal{O}_1$ transversely,
	$\mathcal{L}$ is fully carried by $B$.
	And 
	$M \setminus \setminus B \cong M - Int(N(B)) = C \times I$
	is a compact $3$-ball.
	Hence Proposition \ref{3-ball} holds.
\end{proof}

Henceforth,
we will always denote $\mathcal{N}$ by $N(B)$ and assume
$\{\text{interval fibers of } N(B)\} = \mathcal{O}_1$.
We note that from the proof of Proposition \ref{3-ball},
$N(B)$ has the following properties:
$\partial_h N(B)$ is the union of two disks contained in $\lambda \times \{0\}, \lambda \times \{1\}$ respectively,
and
$\partial_v N(B)$ is a properly embedded annulus in $\lambda \times I$.
Moreover,
$\mathcal{F}_2 \mid_{N(B)}$ is a foliation of $N(B)$ such that:
the interior of $\partial_h N(B)$ is tangent to $\mathcal{F}_2 \mid_{N(B)}$,
and
the interior of $\partial_v N(B)$ is transverse to $\mathcal{F}_2 \mid_{N(B)}$.

Let $\widetilde{\mathcal{F}_2}$ be the pull-back foliation of $\mathcal{F}_2$ in
$\widetilde{M}$.
Let $\widetilde{\lambda} \times I$ be a component of $p^{-1}(\lambda \times I)$,
where $p(\widetilde{\lambda} \times \{t\}) = \lambda \times \{t\}$ for all $t \in I$.
Let $\widetilde{\lambda_0} = \widetilde{\lambda} \times \{0\}$ and
$\widetilde{\lambda_1} = \widetilde{\lambda} \times \{1\}$.
Since $\lambda \times \{0\}, \lambda \times \{1\}$ are $2$-plane leaves,
we have
$Stab_G(\widetilde{\lambda_0}) = 
Stab_G(\widetilde{\lambda_1}) = 1$.

\begin{defn}\rm
	We denote by $L_2$ the leaf space of $\widetilde{\mathcal{F}_2}$.
	Let $i_{des}: L_2 \to \mathbb{R}$ be an orientation-preserving homeomorphism.
	Let
	$K_0 = i_{des}(\widetilde{\lambda_0})$,
	$K_1 = i_{des}(\widetilde{\lambda_1})$.
	Let
	$\{r_g: \mathbb{R} \to \mathbb{R} \mid g \in G\}$ be the action of
	$G$ on $\mathbb{R}$ defined by
	$r_g = i_{des} \circ t_g \circ i^{-1}_{des}$ for any
	$g \in G$.
\end{defn}

\begin{remark}\rm\label{positively oriented transversal}
	Since $i_{des}$ is orientation-preserving,
	$i_{des}$ projects positively oriented transversals of $\widetilde{\mathcal{F}_2}$ to
	positively oriented intervals in $\mathbb{R}$.
	So $K_1 > K_0$.
	Notice that $K_0$, $K_1$ have trivial stabilizers under the action
	$\{r_g: \mathbb{R} \to \mathbb{R} \mid g \in G\}$,
	and all blown-up regions
	$\{t_g(\widetilde{\lambda} \times I)\}_{g \in G}$
	are disjoint in $\widetilde{M}$.
	Hence the images of
	$[K_0,K_1]$ under the action $\{r_g \mid g \in G\}$ are pairwise disjoint.
\end{remark}

\begin{defn}\rm\label{<}
	Let ``$<$'' be the strict linear order of $G$ defined by:
	
	$\bullet$
	For any distinct elements $g,h \in G$,
	$g < h$ if
	$r_g(K_0) < r_h(K_0)$.
\end{defn}

\begin{fact}\rm\label{left-invariant}
	``$<$'' is a left-invariant order of $G$.
\end{fact}
\begin{proof}
	Suppose that $g,h,q \in G$ and $g < h$.
	Then
	$$r_{qg}(K_0)
	= r_{q}(r_{g}(K_0))
	< r_{q}(r_{h}(K_0))
	= r_{qh}(K_0).$$
	So ``$<$'' is left-invariant.
\end{proof}

Let $\widetilde{B} = p^{-1}(B)$,
and let
$Sec(\widetilde{B}) = \{\text{branch sectors of }\widetilde{B}\}$.
As in Remark \ref{pull-back branched surface},
we denote by
$\widetilde{N(B)} = p^{-1}(N(B))$ the fibered neighborhood of $\widetilde{B}$
and denote by
$\widetilde{\pi}: \widetilde{N(B)} \to \widetilde{B}$
the collapsing map for
$\widetilde{N(B)}$.
Then $\widetilde{\mathcal{F}_2} \mid_{\widetilde{N(B)}}$ is 
a foliation of $\widetilde{N(B)}$ such that:
the interior of $\partial_h \widetilde{N(B)}$ is tangent to $\widetilde{\mathcal{F}_2} \mid_{\widetilde{N(B)}}$,
and
the interior of $\partial_v \widetilde{N(B)}$ is transverse to $\widetilde{\mathcal{F}_2} \mid_{\widetilde{N(B)}}$.

Let $\Gamma$ be the fundamental domain of $M$ in $\widetilde{M}$
such that
$p(\partial \Gamma) = B$.
Without loss of generality,
we assume that the base point $\widetilde{x}$ is contained in $Int(\Gamma)$.

\begin{lm}\label{endpoints}
	Let $F \in Sec(\widetilde{B})$.
	Let $I_0$ be an interval fiber of $\widetilde{N(B)}$ 
	such that $\widetilde{\pi}(I_0) \in Int(F)$.
	Let $h,s \in G$ for which
	$F$ is a common face of $t_h(\Gamma)$, $t_s(\Gamma)$ and
	$h > s$.
	
	(a)
	The two endpoints of $I_0$ are contained in
	$t_h(\widetilde{\lambda_0}),
	t_s(\widetilde{\lambda_1})$ respectively
	(cf. Figure \ref{descend}).
	
	(b)
	Assume that $I_0$ has a direction which
	starts at the endpoint contained in $t_s(\widetilde{\lambda_1})$ and
	ends at the endpoint contained in
	$t_h(\widetilde{\lambda_0})$.
	Then $I_0$ is a positively oriented transversal of $\widetilde{\mathcal{F}_2}$.
\end{lm}
\begin{proof}
	Let $\mathcal{G} = M - Int(N(B))$ and
	$\widetilde{\mathcal{G}} = p^{-1}(\mathcal{G}) \cap \Gamma$.
	Then the two endpoints of $I_0$ are contained in 
	$t_s(\widetilde{\mathcal{G}})$,
	$t_h(\widetilde{\mathcal{G}})$ respectively.
	Thus, 
	one endpoint of $I_0$ is contained in
	$t_s(\widetilde{\lambda_0}) \cup t_s(\widetilde{\lambda_1})$,
	and the other endpoint of $I_0$ is contained in
	$t_h(\widetilde{\lambda_0}) \cup t_h(\widetilde{\lambda_1})$.
	
	Since distinct elements of 
	$\{r_g\}_{g \in G}$
	take $[K_0,K_1]$ to
	disjoint intervals in $\mathbb{R}$,
	we have
	$r_h(K_1) > r_h(K_0) > r_s(K_1) > r_s(K_0)$,
	and thus
	$$i_{des}(t_h(\widetilde{\lambda_1})) >
	i_{des}(t_h(\widetilde{\lambda_0})) >
	i_{des}(t_s(\widetilde{\lambda_1})) >
	i_{des}(t_s(\widetilde{\lambda_0})).$$
	It follows that the two endpoints of $I_0$ are contained in 
	$t_h(\widetilde{\lambda_0}),
	t_s(\widetilde{\lambda_1})$ respectively
	(cf. Figure \ref{descend}).
	Therefore, 
	(a) holds.
	
	Since $i_{des}(t_h(\widetilde{\lambda_0})) >
	i_{des}(t_s(\widetilde{\lambda_1}))$,
	every transversal of $\widetilde{\mathcal{F}_2}$ from
	$t_s(\widetilde{\lambda_1})$ to
	$t_h(\widetilde{\lambda_0})$ is positively oriented. So (b) holds.
\end{proof}

Recall that Subsection \ref{subsection 3.2} provides 
a process to construct a branched surface given by
the order-domain pair $(<,\Gamma)$.
We denote this branched surface by $B^{*}$.
We show $B^{*} = B$ as follows.

\begin{figure}\label{descend}
	\includegraphics[width=0.7\textwidth]{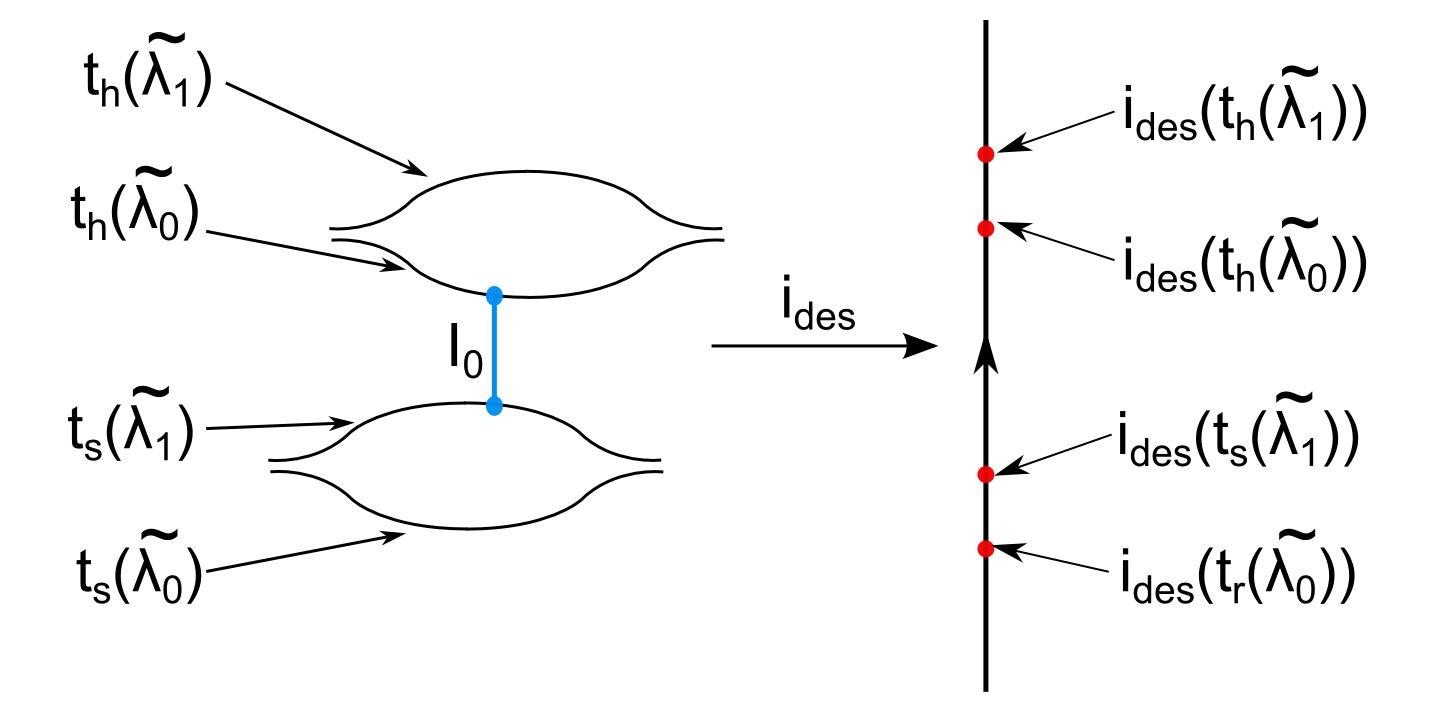}
	\caption{The two endpoints of $I_0$ are at
		$t_h(\widetilde{\lambda_0}),
		t_s(\widetilde{\lambda_1})$ respectively.}
\end{figure}

\begin{prop}\label{resulting branched surface}
	$B^{*} = B$.
\end{prop}
\begin{proof}
	Let $\widetilde{B^{*}} = p^{-1}(B^{*})$.
	Since the branch sectors of $B, B^{*}$ are same,
	we have
	$$\{\text{branch sectors of }\widetilde{B^{*}}\} = Sec(\widetilde{B}).$$
	We compare the co-orientations for $B, B^{*}$ on their branch sectors as follows.
	
	We assume that every interval fiber of $\widetilde{N(B)}$ has
	an orientation so that it is a positively oriented transversal of $\widetilde{\mathcal{F}_2}$.
	$\widetilde{B}$ has
	a co-orientation on $Sec(\widetilde{B})$ such that
	every $F \in Sec(\widetilde{B})$ has an orientation that is consistent with
	the orientation of the interval fibers of $\widetilde{N(B)}$ on $F$,
	i.e. is consistent with the direction of positively oriented transversals of $\widetilde{\mathcal{F}_2}$.
	
	Let $F \in Sec(\widetilde{B})$.
	We denote by $h,s \in G$ for which
	$F$ is a common face of $t_h(\Gamma)$, $t_s(\Gamma)$ and
	$h > s$.
	By Definition \ref{orienting},
	$F$ has positive orientation with respect to $t_s(\Gamma)$ as a branch sector of
	$\widetilde{B^{*}}$.
	By Lemma \ref{endpoints} (b),
	$F$ also has positive orientation with respect to $t_s(\Gamma)$ as a branch sector of
	$\widetilde{B}$.
	Therefore,
	$B^{*} = B$.
\end{proof}

Notice that $\partial_v N(B)$ is connected.
Combining Proposition \ref{resulting branched surface} and 
Proposition \ref{extendability} (1),
resulting foliations of $(<,\Gamma)$ in $M - Int(B^{3})$ can extend to taut foliations in $M$.
Furthermore,

\begin{prop}\label{resulting}
	There is a resulting foliation $\mathcal{F}$ of $(<,\Gamma)$ in $N(B) \cong M - Int(B^{3})$ such that
	$\mathcal{F} = \mathcal{F}_{2} \mid_{N(B)}$.
\end{prop}
\begin{sketch of the proof}\rm
	Recall from Section \ref{section 4}, 
	we first choose an action
	$\{\rho_g: \mathbb{R} \to \mathbb{R} \mid g \in G\}$ and 
	an interval $[N_0,N_1]$ (cf. Definition \ref{rho}),
	and then use them to 
	construct a resulting foliation of $(<,\Gamma)$.
	We set 
	$\{\rho_g: \mathbb{R} \to \mathbb{R} \mid g \in G\}$ to be
	$\{r_g: \mathbb{R} \to \mathbb{R} \mid g \in G\}$
	and set $[N_0, N_1]$ to be $[K_0, K_1]$.
	Then we have a resulting foliation $\mathcal{F}$ of $(<,\Gamma)$ by
	the process of Section \ref{section 4}.
	And we can verify 
	$\mathcal{F} = \mathcal{F}_{2} \mid_{N(B)}$.
	See the following proof for details.
\end{sketch of the proof}
\begin{proof}
	To begin with,
	we review some ingredients in Section \ref{section 4}:
	
	\begin{ingredient}\label{ingredient 1}\rm
	In Definition \ref{rho},
	we choose an action
	$\{\rho_g: \mathbb{R} \to \mathbb{R} \mid g \in G\}$ 
	of $G$ on $\mathbb{R}$ and a closed interval 
	$[N_0,N_1] \subseteq \mathbb{R}$ such that
	(1)
	$[N_0,N_1]$ has pairwise disjoint images under
	the action $\{\rho_g\}_{g \in G}$,
	(2)
	$\rho_g(N_0) < \rho_h(N_0)$ for 
	any $g,h \in G$ with $g < h$.
	\end{ingredient}
	
	\begin{ingredient}\label{ingredient 2}\rm
	There is a map
	$i_0: \coprod_{F \in Sec(\widetilde{B})} (F \times I) \to \mathbb{R}$ (cf. Definition \ref{piece value})
	with the following properties:
	
	(1)
	Let $F \in Sec(\widetilde{B})$, $t \in I$.
	For arbitrary distinct $m,n \in F \times \{t\}$,
	we have
	$i_0(m) = i_0(n)$.
	
	(2)
	For all $F \in Sec(\widetilde{B})$, $t \in I$,
	we have
	$\rho_g(i_0(F \times \{t\})) = i_0(t_g(F) \times \{t\})$
	(cf. Definition \ref{piece value}, Fact \ref{transform}).
	
	(3)
	Let $F \in Sec(\widetilde{B})$.
	We denote by $u,v \in G$ for which
	$F$ is a common face of $t_u(\Gamma),t_v(\Gamma)$ and
	$v > u$.
	Then
	$i_0(F \times \{0\}) = \rho_u(N_1)$,
	$i_0(F \times \{1\}) = \rho_v(N_0)$
	(cf. Definition \ref{I_F}, Definition \ref{psi}).
    \end{ingredient}
	
	\begin{ingredient}\label{ingredient 3}\rm
	We decompose $\widetilde{N(B)}$ to
	$\coprod_{F \in Sec(\widetilde{B})} (F \times I)$
	and then assign a gluing to it
	so that the horizontal sheets 
	$\{F \times \{t\} \mid 
	F \in Sec(\widetilde{B}), t \in I\}$ are glued to
	a $\pi_1$-equivariant foliation in $\widetilde{N(B)}$.
	For arbitrary $a,b \in I$ and distinct $F,X \in Sec(\widetilde{B})$
	with $F \cap X \ne \emptyset$,
	$F \times \{a\}$, $X \times \{b\}$ are attached when 
	$i_0(F \times \{a\}) = i_0(X \times \{b\})$
	(cf. Fact \ref{invariant}).
    \end{ingredient}
	
	We set
	(1)
	$\rho_g = r_g$ for every $g \in G$,
	(2)
	$N_0 = K_0$,
	$N_1 = K_1$.
	Following the process in Section \ref{section 4},
	we can
	construct a $\pi_1$-equivariant foliation $\widetilde{\mathcal{F}}$ in $\widetilde{N(B)}$,
	which descends to a foliation
	$\mathcal{F}$ of $N(B)$.
	We verify 
	$\widetilde{\mathcal{F}} = 
	\widetilde{\mathcal{F}_2} \mid_{\widetilde{N(B)}}$ in the following.
	
	Similar to Subsection \ref{subsection 4.4}, 
	we can decompose
	$\widetilde{N(B)}$ to
	$\coprod_{F \in Sec(\widetilde{B})} (F \times I)$.
	Then $\widetilde{\mathcal{F}_2} \mid_{\widetilde{N(B)}}$ is
	decomposed to a collection of horizontal sheets
	$\{F \times \{t\} \mid F \in Sec(\widetilde{B}), t \in I\}$.
	And we assume that: 
	for every $F \in Sec(\widetilde{B})$,
	a transversal of $\widetilde{\mathcal{F}_2}$ from $F \times \{0\}$ to $F \times \{1\}$ is
	a positively oriented.
	Let $i_{proj}: \widetilde{M} \to \mathbb{R}$ be the combination
	$\widetilde{M}\stackrel{proj}{\to} L_2 \stackrel{i_{des}}{\to} \mathbb{R}$,
	where $\widetilde{M} \stackrel{proj}{\to} L_2$ denotes the projection map from
	$\widetilde{M}$ to the leaf space $L_2$ of $\widetilde{\mathcal{F}_2}$.
	
	Let $F \in Sec(\widetilde{B})$.
	We denote by $u,v \in G$ for which
	$F$ is a common face of $t_u(\Gamma),t_v(\Gamma)$ and
	$v > u$.
	Recall from Lemma \ref{endpoints},
	$$i_{proj}(F \times \{0\}) = i_{des}(t_u(\widetilde{\lambda_1})) = r_u(K_1),$$
	$$i_{proj}(F \times \{1\}) = i_{des}(t_v(\widetilde{\lambda_0})) = r_v(K_0).$$
	By Ingredient \ref{ingredient 2} (c),
	we have
	$i_{proj}(F \times \{0\}) = i_0(F \times \{0\})$,
	$i_{proj}(F \times \{1\}) = i_0(F \times \{1\})$.
	We can assume that the $t$-variable of
	$\{F \times \{t\} \mid t \in I\}$ is parameterized such that
	$$i_{proj}(F \times \{t\}) = i_0(F \times \{t\}).$$
	By Ingredient \ref{ingredient 2} (b),
	we have
	$t_g(F \times \{t\}) = t_g(F) \times \{t\}$ for every $g \in G$.
	
	At last,
	we glue $\coprod_{F \in Sec(\widetilde{B})} (F \times I)$ to 
	$\widetilde{N(B)}$ that make
	$\{F \times \{t\} \mid F \in Sec(\widetilde{B}), t \in I\}$ be glued to
	$\widetilde{\mathcal{F}_2} \mid_{\widetilde{N(B)}}$.
	To show $\widetilde{\mathcal{F}_2} \mid_{\widetilde{N(B)}} =
	\widetilde{\mathcal{F}}$,
	we only need to verify that:
	the gluing relation for this gluing
	is same as the gluing relation for
	$\widetilde{\mathcal{F}}$ as in Ingredient \ref{ingredient 3}.
	For any $F,X \in Sec(\widetilde{B})$ with
	$F \cap X \ne \emptyset$ and 
	arbitrary $a,b \in I$,
	the two horizontal sheets
	$F \times \{a\}, X \times \{b\}$ are contained in 
	the same leaf of $\widetilde{\mathcal{F}_2}$ 
	if and only if 
	$i_{proj}(F \times \{a\}) = i_{proj}(X \times \{b\})$,
	and therefore if and only if
	$i_0(F \times \{a\}) = i_0(X \times \{b\})$.
	This is same as the gluing relation for
	$\widetilde{\mathcal{F}}$ in Ingredient \ref{ingredient 3}.
	So 
	$\widetilde{\mathcal{F}} = 
	\widetilde{\mathcal{F}_2} \mid_{\widetilde{N(B)}}$,
	and therefore
	$\mathcal{F} = \mathcal{F}_2 \mid_{N(B)}$.    
\end{proof}

Now we already have operations
$\mathcal{F}_0 \rightsquigarrow 
\mathcal{F}_1 \rightsquigarrow 
\mathcal{F}_2 \rightsquigarrow \mathcal{F}$.
Since
$\mathcal{F} = \mathcal{F}_{2} \mid_{N(B)}$ is a resulting foliation of $(<,\Gamma)$ 
and $\mathcal{F}_0$ can be obtained from 
doing a collapsing operation for $\mathcal{F}_2$,
$\mathcal{F} \rightsquigarrow \mathcal{F}_2 \rightsquigarrow  \mathcal{F}_0$
is the process as required in Theorem \ref{process}.

\end{document}